\documentclass[12pt]{amsart}

\usepackage{etex}
\usepackage{amsmath, amssymb}
\usepackage{array}
\usepackage[frame,cmtip,arrow,matrix,line,graph,curve]{xy}
\usepackage{graphpap, color, paralist, pstricks}
\usepackage[mathscr]{eucal}
\usepackage[pdftex]{graphicx}
\usepackage[pdftex,colorlinks,backref=page,citecolor=blue]{hyperref}
\usepackage{tikz}

\newtheorem{theorem}{Theorem}[section]
\newtheorem{proposition}[theorem]{Proposition}
\newtheorem{corollary}[theorem]{Corollary}
\newtheorem{lemma}[theorem]{Lemma}

\newtheorem{conjecture}[theorem]{Conjecture}

\theoremstyle{definition}
\newtheorem{definition}[theorem]{Definition}

\newtheorem{remark}[theorem]{Remark}

\newcommand{\PP}{\mathbb{P}}
\newcommand{\QQ}{\mathbb{Q}}
\newcommand{\CC}{\mathbb{C}}
\newcommand{\RR}{\mathbb{R}}
\newcommand{\ZZ}{\mathbb{Z}}

\newcommand{\VV}{\mathbb{V}}

\newcommand{\cO}{\mathcal{O} }

\newcommand{\cE}{\mathcal{E} }
\newcommand{\cF}{\mathcal{F} }
\newcommand{\cG}{\mathcal{G} }

\newcommand{\cJ}{\mathcal{J} }

\newcommand{\cM}{\mathcal{M} }

\newcommand{\cL}{\mathcal{L} }

\newcommand{\cQ}{\mathcal{Q} }

\newcommand{\cT}{\mathcal{T} }

\newcommand{\rH}{\mathrm{H} }
\newcommand{\rM}{\mathrm{M} }
\newcommand{\rN}{\mathrm{N} }

\newcommand{\bE}{\mathbf{E}}
\newcommand{\ba}{\mathbf{a}}
\newcommand{\bfb}{\mathbf{b}}
\newcommand{\bc}{\mathbf{c}}
\newcommand{\bd}{\mathbf{d}}
\newcommand{\bp}{\mathbf{p}}
\newcommand{\bq}{\mathbf{q}}

\newcommand{\proj}{\mathrm{Proj}\;}

\newcommand{\pHom}{\mathrm{ParHom}}
\newcommand{\spHom}{\mathrm{SParHom}}

\newcommand{\cHom}{\mathcal{H}om}

\newcommand{\cpHom}{\mathcal{P}ar\mathcal{H}om}
\newcommand{\cspHom}{\mathcal{SP}ar\mathcal{H}om}

\def\Hom{\mathrm{Hom} }
\def\Ext{\mathrm{Ext} }

\newcommand{\pdeg}{\mathrm{pdeg}\, }
\newcommand{\rk}{\mathrm{rk}\, }

\def\SL{\mathrm{SL}}
\def\sl{\mathfrak{sl}}

\def\Gr{\mathrm{Gr}}
\def\Fl{\mathrm{Fl}}

\def\git{/\!/ }

\def\Eff{\mathrm{Eff} }
\def\St{\mathrm{Stab} }
\def\Pic{\mathrm{Pic} }
\def\Cox{\mathrm{Cox} }

\hypersetup{%
pdftitle={Finite generation of the algebra of type A conformal blocks via birational geometry},%
pdfauthor={Han-Bom Moon, Sang-Bum Yoo},%
pdfkeywords={moduli space, parabolic bundle, conformal block, Mori's program, Mori dream space},%
citecolor=blue,%
linkcolor=blue,%
}

\begin{document}

\title[Finite generation of the algebra of conformal blocks]{Finite generation of the algebra of type A conformal blocks via birational geometry}
\date{\today}

\author{Han-Bom Moon}
\address{Department of Mathematics, Institute for Advanced Study, Princeton, NJ 08540}
\email{hanbommoon@gmail.com}

\author{Sang-Bum Yoo}
\address{Division of General Studies, UNIST, Ulsan 44919, Republic of Korea}
\email{sangbum.yoo@gmail.com}

\begin{abstract}
We study birational geometry of the moduli space of parabolic bundles over a projective line, in the framework of Mori's program. We show that the moduli space is a Mori dream space. As a consequence, we obtain the finite generation of the algebra of type A conformal blocks. Furthermore, we compute the H-representation of the effective cone which was previously obtained by Belkale. For each big divisor, the associated birational model is described in terms of moduli space of parabolic bundles.
\end{abstract}

\maketitle


\section{Introduction}

The aim of this paper is twofold. First of all, we prove the following finiteness theorem.

\begin{theorem}[\protect{Theorem \ref{thm:fingen}}]\label{thm:finitegeneration}
The algebra of type A conformal blocks over a projective line is finitely generated.
\end{theorem}

\emph{Conformal blocks} introduced by Tsuchiya, Kanie, Ueno and Yamada to construct a two-dimensional chiral conformal field theory (WZW model) (\cite{TK88, TUY89, Uen08}). For each $(C, \mathbf{p}) \in \overline{\cM}_{g, n}$, a simple Lie algebra $\mathfrak{g}$, a nonnegative integer $\ell$, and a collection of dominant integral weights $\vec{\lambda} := (\lambda^{1}, \lambda^{2}, \cdots, \lambda^{n})$ such that $(\lambda^{i}, \theta) \le \ell$ where $\theta$ is the highest root, they constructed a finite dimensional vector space $\VV^{\dagger}_{\ell, \vec{\lambda}}$ so-called the space of conformal blocks or the space of vacua. 

Conformal blocks have several interesting connections in algebraic geometry. It is known that $\VV^{\dagger}_{\ell, \vec{\lambda}}$ can be naturally identified with the space of global sections, so-called generalized theta functions, of a certain line bundle on the moduli space of parabolic principal $G$-bundles (\cite{Pau96, LS97}). They can be also regarded as a quantum generalization of invariant factors (\cite{Bel08b}). Recently, conformal blocks have been studied to construct positive vector bundles on $\overline{\cM}_{0,n}$ (see \cite{BGM16} and references therein).

We will focus on $C = \PP^{1}$ and $\mathfrak{g} = \sl_{r}$ case. There is a map
\[
	\VV^{\dagger}_{\ell, \vec{\lambda}} \otimes
	\VV^{\dagger}_{m, \vec{\mu}}
	\to \VV^{\dagger}_{\ell+m, \vec{\lambda}+\vec{\mu}}
\]
which defines a commutative $\mathrm{Pic}(\rM_{\bp}(r, 0))$-graded $\CC$-algebra structure on
\[
	\VV^{\dagger} :=
	\bigoplus_{\ell, \vec{\lambda}}\VV^{\dagger}_{\ell, \vec{\lambda}}
\]
where $\rM_{\bp}(r,0)$ is the moduli stack of rank $r$, degree $0$ parabolic bundles on $\PP^{1}$. $\VV^{\dagger}$ is called the \emph{algebra of conformal blocks} and naturally identified with the Cox ring of $\rM_{\bp}(r, 0)$.

Despite of many results on the structure of $\VV^{\dagger}$, many fundamental questions such as the finite generation of $\VV^{\dagger}$, are still open. See \cite{Man09b, Man13, Man16, MY16} for some partial results. To prove the finite generation of $\VV^{\dagger}$, we study birational geometry of $\rM_{\bp}(r, 0, \ba)$, the moduli space of rank $r$, degree $0$, $\ba$-semistable parabolic bundles on $\PP^{1}$, in the framework of \emph{Mori's program}.

For a normal $\QQ$-factorial projective variety $X$ with trivial irregularity, Mori's program, or log minimal model program, consists of the following three steps.

\begin{enumerate}
\item Compute the effective cone $\Eff(X)$ in $\rN^{1}(X)_{\RR}$.
\item For each integral divisor $D \in \Eff(X)$, compute the projective model
\[
	X(D) := \proj \bigoplus_{m \ge 0}\rH^{0}(X, \cO(mD)).
\]
\item Study the rational contraction $X \dashrightarrow X(D)$.
\end{enumerate}

However, even among very simple varieties such as a blow-up of $\PP^{n}$ along some points, there are examples that Mori's program cannot be completed because of two kinds of infinity: There may be infinitely many rational contractions, and more seriously, the section ring $\bigoplus_{m \ge 0}\rH^{0}(X, \cO(mD))$ may not be finitely generated and thus $X(D)$ is not a projective variety.

A \emph{Mori dream space} (\cite{HK00}), MDS for short, is a special kind of variety that has no such technical difficulties to run Mori's program: The effective cone is polyhedral, there is a finite chamber structure which corresponds to finitely many different projective models, and every divisor has a finitely generated section ring. In Section \ref{sec:MDS}, we show that $\rM_{\bp}(r, 0, \ba)$ is a MDS. On the other hand, under some assumption, $\VV^{\dagger}$ can be identified with $\Cox(\rM_{\bp}(r, 0, \ba))$. We obtain the finite generation by \cite[Proposition 2.9]{HK00}.

Even for a MDS $X$, it is very difficult to complete Mori's program if $\rho(X) \ge 3$ (see \cite{CM17} for an example). As the second main result, we complete Mori's program for $\rM_{\bp}(r, 0, \ba)$. The rank two case was done in \cite{MY16}.

\begin{theorem}\label{thm:Moriprogram}
Assume $n > 2r$. Let $\rM := \rM_{\bp}(r, 0, \ba)$ for a general effective parabolic weight $\ba$.
\begin{enumerate}
\item (Proposition \ref{prop:facetGWinv}) If $\rho(\rM) = (r-1)n + 1$, which is the possible maximum, the effective cone $\Eff(\rM)$ is the intersection of an explicit finite set of half-planes in $\rN^{1}(\rM)_{\RR}$. In particular, for any collection of partitions $\lambda^{1}, \cdots, \lambda^{n}$, $d \ge 0$ and $0 < s < r$ such that the Gromov-Witten invariant $\langle \omega_{\lambda^{1}}, \omega_{\lambda^{2}}, \cdots, \omega_{\lambda^{n}}\rangle_{d}$ of Grassmannian $\Gr(s, r)$ is one, there is a hyperplane supporting $\Eff(\rM)$. \item (Proposition \ref{prop:birationalmodels}, Section \ref{ssec:projmodel}) For any $D \in \mathrm{int} \Eff(\rM)$, $\rM(D) = \rM_{\bp}(r, 0, \bfb)$ for some parabolic weight $\bfb$. The projective models associated to facets of $\Eff(\rM)$ can be also described in terms of moduli spaces of parabolic bundles.
\item For a general $D \in \mathrm{int}\Eff(\rM)$, the rational contraction $\rM \dashrightarrow \rM(D)$ is a composition of smooth blow-ups and blow-downs.
\end{enumerate}
\end{theorem}

\begin{remark}
\begin{enumerate}
\item When $n$ is small, we may think of $\rM$ as the target of an algebraic fiber space $\rM_{\bq}(r, 0, \ba') \to \rM$ where the domain is the moduli space of parabolic bundles with larger number of parabolic points. Thus the Mori's program for $\rM$ becomes a part of that of $\rM_{\bq}(r, 0, \ba')$.
\item The H-representation in Item (1) was obtained by Belkale in \cite[Theorem 2.8]{Bel08b} in a greater generality by a different method. Our approach using wall-crossings is independent from his idea and is elementary. On the other hand, his result indeed tells us a strong positivity: Any integral divisor class in $\Eff(\rM)$ is effective.
\item Once (2) is shown, (3) follows from a result of Thaddeus (\cite[Section 7]{Tha96}).
\end{enumerate}
\end{remark}

We leave a brief outline of the proof of Theorem \ref{thm:finitegeneration}. For simplicity, suppose that $n$ is large enough. The main technique we employ is a careful analysis of wall-crossings. The moduli space $\rM_{\bp}(r, 0, \ba)$ depends on a choice of a parabolic weight $\ba$. For two general parabolic weights $\ba$ and $\bfb$, there is a birational map $\rM_{\bp}(r, 0, \ba) \dashrightarrow \rM_{\bp}(r, 0, \bfb)$ provided stable loci of these two moduli spaces are nonempty. Then the birational map can be decomposed into finitely many explicit blow-ups and blow-downs, and the change can be measured explicitly. Furthermore, when $\ba$ is sufficiently small, then $\rM_{\bp}(r, 0, \ba) \cong \Fl(V)^{n}\git_{L}\SL_{r}$ for some explicit linearization $L$. By analysing the geometry of the GIT quotient and wall-crossings, we obtain the canonical divisor of $\rM_{\bp}(r, 0, \ba)$ where $\ba$ is dominant in the sense that $\Cox(\rM_{\bp}(r, 0)) = \Cox(\rM_{\bp}(r, 0, \ba))$. We show that after finitely many flips, the anticanonical divisor $-K$ becomes big and nef. Because the moduli space is smooth, it is of Fano type. Therefore it is a MDS.

The method of the proof does not provide any explicit set of generators. It is an interesting problem to construct a generating set. Several results of Manon (\cite{Man09b, Man13, Man16}) suggest the next conjecture. Note that the effective cone is polyhedral, so there are finitely many extremal rays.

\begin{conjecture}
The algebra $\VV^{\dagger}$ of conformal blocks is generated by the set of effective divisors whose numerical classes are the first integral points of extremal rays of $\Eff(\rM_{\bp}(r, 0))$. 
\end{conjecture}

\subsection{Organization of the paper}

We begin in Section \ref{sec:modparbdl} by recalling some preliminaries on the moduli space of parabolic bundles and deformation theory. In Section \ref{sec:smallweight}, we identify moduli spaces of parabolic bundles with small weights with elementary GIT quotients. Section \ref{sec:wallcrossing} reviews wall-crossing analysis. In Section \ref{sec:MDS}, we prove Theorem \ref{thm:finitegeneration}. Finally in Section \ref{sec:MMP}, we prove Theorem \ref{thm:Moriprogram}.

\subsection*{Notations and conventions}

We work on an algebraically closed field $\CC$ of characteristic zero. Unless there is an explicit statement, we will fix $n$ distinct points $\bp = (p^{1}, \cdots, p^{n})$ on $\PP^{1}$. These points are called \emph{parabolic points}. The number $n$ of parabolic points is always at least three. $[r]$ denotes the set $\{1,2,\cdots,r\}$. To minimize the introduction of cumbersome notation, mainly we will discuss parabolic bundles with full flags only, except Section \ref{ssec:projmodel} on degenerations of moduli spaces. The readers may easily generalize most part of the paper, to the partial flag cases. In many literatures the dual $\VV_{\ell, \vec{\lambda}}$ of $\VV_{\ell, \vec{\lambda}}^{\dagger}$ has been denoted by conformal blocks.

\subsection*{Acknowledgements}

The authors thank Prakash Belkale and Young-Hoon Kiem for their comments and suggestions on an earlier draft. The first author thanks Jinhyung Park for helpful conversations.


\section{Moduli space of parabolic bundles}\label{sec:modparbdl}

In this section, we give a brief review on parabolic bundles and their moduli spaces.

\subsection{Parabolic bundles and their moduli spaces}

\begin{definition}
Fix parabolic points $\mathbf{p} = (p^{1}, p^{2}, \cdots, p^{n})$ on $\PP^{1}$. A rank $r$ \emph{quasi parabolic bundle} over $(\PP^{1}, \mathbf{p})$ is a collection of data $\cE := (E, \{W_{\bullet}^{i}\})$ where:
\begin{enumerate}
\item $E$ is a rank $r$ vector bundle over $\PP^{1}$;
\item For each $1 \le i \le n$, $W_{\bullet}^{i} \in \Fl(E|_{p^{i}})$. In other words, $W_{\bullet}^{i}$ is a strictly increasing filtration of subspaces $0 \subsetneq W_{1}^{i} \subsetneq W_{2}^{i} \subsetneq \cdots \subsetneq W_{r-1}^{i} \subsetneq W_{r}^{i} = E|_{p^{i}}$. In particular, $\dim W_{j}^{i} = j$.
\end{enumerate}
\end{definition}

Let $\rM_{\bp}(r, d)$ be the moduli stack of rank $r$, degree $d$ quasi parabolic bundles over $(\PP^{1}, \bp)$. This moduli stack is highly non-separated Artin stack. To obtain a proper moduli space, as in the case of moduli spaces of ordinary vector bundles, we introduce the parabolic slope and a stability condition, and collect semi-stable objects only. One major difference here is that there are many different ways to define stability while there is a standard one in the case of ordinary bundles.

\begin{definition}
\begin{enumerate}
\item A \emph{parabolic weight} is a collection $\ba = (a_{\bullet}^{1}, a_{\bullet}^{2}, \cdots, a_{\bullet}^{n})$ of strictly decreasing sequences $a_{\bullet}^{i} = (1 > a_{1}^{i} > \cdots > a_{r-1}^{i} > a_{r}^{i} \ge 0)$ of length $r$. Let $|\ba|_{j} := \sum_{i=1}^{n}a_{j}^{i}$ and $|\ba| := \sum_{j=1}^{r-1}|\ba|_{j}$.
\item A \emph{parabolic bundle} is a collection $\cE := (E, \{W_{\bullet}^{i}\}, \ba)$.
\end{enumerate}
\end{definition}

\begin{definition}
Let $\cE = (E, \{W_{\bullet}^{i}\}, \ba)$ be a parabolic bundle.
\begin{enumerate}
\item The \emph{parabolic degree} of $\cE$ is
\[
	\pdeg \cE := \deg E + |\ba|.
\]
\item The \emph{parabolic slope} of $\cE$ is $\mu(\cE) = \pdeg \cE/\rk E$.
\end{enumerate}
\end{definition}

Let $\cE = (E, \{W_{\bullet}^{i}\}, \ba)$ be a parabolic bundle of rank $r$. For each subbundle $F \subset E$, there is a natural induced flag structure ${W|_{F}}_{\bullet}^{i}$ on $F|_{p^{i}}$. More precisely, let $\ell$ be the smallest index such that $\dim(W_{\ell}^{i} \cap F|_{p^{i}}) = j$. Then ${W|_{F}}_{j}^{i} = W_{\ell}^{i}\cap F|_{p^{i}}$. Furthermore, we can define the induced parabolic weight $\bfb = (b_{\bullet}^{i})$ on $F|_{p^{i}}$ as $b_{j}^{i} = a_{\ell}^{i}$. This collection of data $\cF := (F, \{{W|_{F}}_{\bullet}^{i}\}, \bfb)$ is called a \emph{parabolic subbundle} of $\cE$. Similarly, one can define the induced flag $W/F_{\bullet}^{i}$ on $E/F|_{p^{i}}$, the inherited parabolic weight $\bc$, and the \emph{quotient parabolic bundle} $\cQ := (E/F, \{W/F_{\bullet}^{i}\},\bc)$.

\begin{definition}
A parabolic bundle $\cE = (E, \{W_{\bullet}^{i}\}, \ba)$ is $\mathbf{a}$-\emph{(semi)-stable} if for every parabolic subbundle $\cF$ of $\cE$, $\mu(\cF) (\le) < \mu(\cE)$.
\end{definition}

Let $\rM_{\bp}(r, d, \ba)$ be the moduli space of S-equivalent classes of rank $r$, topological degree $d$ $\ba$-semistable parabolic bundles over $(\PP^{1}, \mathbf{p})$. It is an irreducible normal projective variety of dimension $n{r \choose 2} - r^{2}+1$ if it is nonempty (\cite[Theorem 4.1]{MS80}).

\begin{remark}
For a general parabolic weight $\ba$ for a rank $r$ quasi parabolic bundle, we can define the \emph{normalized} weight $\ba'$ as ${a'}_{j}^{i} = a_{j}^{i} - a_{r}^{i}$. Then it is straightforward to check that the map
\begin{eqnarray*}
	\rM_{\bp}(r, d, \ba) & \to & \rM_{\bp}(r, d, \ba')\\
	(E, \{W_{\bullet}^{i}\}, \ba) & \mapsto &
	(E, \{W_{\bullet}^{i}\}, \ba')
\end{eqnarray*}
is an isomorphism. Thus we assume that any given parabolic weight is normalized.
\end{remark}

Finally, we leave two notions on weight data.

\begin{definition}
\begin{enumerate}
\item A parabolic weight $\ba$ is \emph{effective} if $\rM_{\bp}(r, d, \ba)$ is nonempty. 
\item A parabolic weight $\ba$ is \emph{general} if the $\ba$-semistability coincides with the $\ba$-stability.
\end{enumerate}
\end{definition}

\begin{remark}
The notion of parabolic bundles can be naturally generalized to parabolic bundles with \emph{partial flags}. For notational simplicity, we do not describe them here. Consult \cite[Definition 1.5]{MS80}. In this paper, we use moduli spaces of parabolic bundles with partial flags only in Section \ref{ssec:projmodel} to describe projective models associated to non-big divisors. A reader who is not interested in this topic can ignore partial flag cases.
\end{remark}

\subsection{The algebra of conformal blocks}\label{ssec:conformalblock}

For an $r$-dimensional vector space $V$, the full flag variety $\Fl(V)$ is embedded into $\prod_{j=1}^{r-1}\Gr(j, V)$, and every line bundle on $\Fl(V)$ is the restriction of $\cO(b_{\bullet}) := \cO(b_{1}, b_{2}, \cdots, b_{r-1})$. By Borel-Weil theorem, if all $b_{i}$'s are nonnegative, or equivalently $\cO(b_{\bullet})$ is effective, then $\rH^{0}(\Fl(V), \cO(b_{\bullet}))$ is the irreducible $\SL_{r}$-representation $V_{\lambda}$ with the highest weight $\lambda = \sum_{i=1}^{r-1}b_{r-i}\omega_{i}$. $F_{\lambda}$ denotes $\cO(b_{\bullet})$.

Recall that $\rM_{\bp}(r, 0)$ is the moduli stack of rank $r$, degree $0$ quasi parabolic bundles over $(\PP^{1}, \mathbf{p})$. The Picard group of $\rM_{\bp}(r, 0)$ is isomorphic to
\[
	\ZZ \cL \times \prod_{i=1}^{n}\mathrm{Pic}(\Fl(V))
\]
(\cite{LS97}). In particular, its Picard number is $(r-1)n + 1$.

The generator $\cL$ is the determinant line bundle on $\rM_{\bp}(r, 0)$ which has the following functorial property: For any family of rank $r$ quasi parabolic bundles $\cE = (E, \{W_{\bullet}^{i}\})$ over $S$, consider the determinant bundle $L_{S} := \det R^{1}\pi_{S*}E \otimes (\det \pi_{S*}E)^{-1}$, where $\pi_{S}:X\times S\to S$ is the projection to $S$. If $p : S \to \rM_{\bp}(r, 0)$ is the functorial morphism, then $p^{*}(\cL) = L_{S}$. The line bundle $\cL$ has a unique section denoted by $\Theta$. This section $\Theta$ vanishes exactly on the locus of $\cE = (E, \{W_{\bullet}^{i}\})$ such that $E \ne \cO^{r}$ (\cite[Section 10.2]{BGM15}).

Any line bundle $F \in \mathrm{Pic}(\rM_{\bp}(r, 0))$ can be written uniquely as $\cL^{\ell}\otimes \otimes_{i=1}^{n}F_{\lambda^{i}}$ where $F_{\lambda^{i}}$ is a line bundle associated to the integer partition $\lambda^{i}$. The space of global sections $\rH^{0}(F)$ is identified with the space of conformal blocks $\VV_{\ell, \vec{\lambda}}^{\dagger} := \VV_{\ell,(\lambda^{1}, \cdots, \lambda^{n})}^{\dagger}$ (\cite[Corollary 6.7]{Pau96}). In particular, $\VV_{1, \vec{0}}^{\dagger}$ is generated by $\Theta$. In general, $\VV_{\ell, \vec{\lambda}}^{\dagger}$ is trivial if $\lambda^{i}_{1} > \ell$ for some $1 \le i \le n$.

Note that there is a natural injective map $\VV_{\ell, \vec{\lambda}}^{\dagger} \hookrightarrow \VV_{\ell+1, \vec{\lambda}}^{\dagger}$ given by the multiplication of $\Theta$. Moreover, when $\ell \ge (\sum_{i=1}^{n}\sum_{j=1}^{r-1}\lambda_{j}^{i})/(r+1)$, $\VV_{\ell, \vec{\lambda}}^{\dagger} \cong V_{\vec{\lambda}}^{\SL_{r}} := (\bigoplus_{i=1}^{n}V_{\lambda^{i}})^{\SL_{r}}$ (\cite[Proposition 1.3]{BGM15}). Thus spaces of conformal blocks define a filtration on $V_{\vec{\lambda}}^{\SL_{r}}$ and we may regard them as generalized invariant factors. 

The direct sum of all conformal blocks
\[
	\VV^{\dagger} := \bigoplus_{\ell, \vec{\lambda}}
	\VV_{\ell,\vec{\lambda}}^{\dagger}
\]
has a $\mathrm{Pic}(\rM_{\bp}(r,0))$-graded algebra structure. This algebra $\VV^{\dagger}$ is called the \emph{algebra of conformal blocks}. Note that $\VV^{\dagger}$ is the space of all sections of line bundles on $\rM_{\bp}(r, 0)$. Thus $\VV^{\dagger}$ is the \emph{Cox ring} $\mathrm{Cox}(\rM_{\bp}(r, 0))$ of the moduli stack $\rM_{\bp}(r, 0)$.

\subsection{Deformation theory}

To analyze wall crossings on the moduli space in detail, we employ some results from deformation theory of parabolic bundles, which was intensively studied by Yokogawa in \cite{Yok95}. In this section we summarize some relevant results.

Let $\cE = (E, \{W_{\bullet}^{i}\}, \ba)$ and $\cF = (F, \{{W'}_{\bullet}^{i}\}, \bfb)$ be two parabolic bundles. A bundle morphism $f : E \to F$ is called \emph{(strongly) parabolic} if $f(W_{j}^{i}) \subset {W'}_{k}^{i}$ whenever $a_{j}^{i} (\ge) > b_{k+1}^{i}$. The sheaves of parabolic morphisms and strongly parabolic morphisms are denoted by $\cpHom(\cE, \cF)$ and $\cspHom(\cE, \cF)$ respectively. The spaces of their global sections are denoted by $\pHom(\cE, \cF)$ and $\spHom(\cE, \cF)$.

Yokogawa introduced an abelian category $\mathsf{P}$ of parabolic $\cO_{\PP^1}$-modules which contains the category of parabolic bundles as a full subcategory. $\mathsf{P}$ has enough injective objects, so we can define right derived functor $\Ext^{i}(\cE, -)$ of $\pHom(\cE, -)$. Those cohomology groups can be described in terms of ordinary cohomology groups and behave similarly.

\begin{lemma}[\protect{\cite[Theorem 3.6]{Yok95}}]\label{lem:derivedfunctor}
\[
	\Ext^{i}(\cE, \cF) \cong \rH^{i}(\cpHom(\cE, \cF)).
\]
\end{lemma}

\begin{lemma}[\protect{\cite[Lemma 1.4]{Yok95}}]\label{lem:parabolicextension}
The cohomology $\Ext^{1}(\cE, \cF)$ parametrizes isomorphism classes of parabolic extensions, which are exact sequences $0 \to \cF \to \cG \to \cE \to 0$ in $\mathsf{P}$.
\end{lemma}

Also we have `Serre duality':

\begin{lemma}[\protect{\cite[Proposition 3.7]{Yok95}}]\label{lem:Serreduality}
\[
	\Ext^{1-i}(\cE, \cF \otimes \cO_{\PP^{1}}(n-2)) \cong
	\rH^{i}(\cspHom(\cF, \cE))^{*}.
\]
\end{lemma}

As in the case of the moduli space of ordinary sheaves, if $\cE \in \rM_{\bp}(r, d, \ba)$ is $\mathbf{a}$-stable, then the Zariski tangent space of $\rM_{\bp}(r, d, \ba)$ at $[\cE]$ is $\Ext^{1}(\cE, \cE)$, and if $\Ext^{2}(\cE, \cE) = 0$, then the moduli space is smooth at $[\cE]$ (\cite[Theorem 2.4]{Yok95}). Since $\Ext^{2}(\cE, \cE) \cong \rH^{2}(\cpHom(\cE, \cE))$ is an ordinary sheaf cohomology on a curve, it vanishes. Thus at $[\cE]$, the moduli space is smooth. 

\begin{proposition}\label{prop:smooth}
If $\ba$ is a general parabolic weight, then $\rM_{\bp}(r, d, \ba)$ is a smooth variety.
\end{proposition}


\section{Small weight case}\label{sec:smallweight}

When a parabolic weight $\ba$ is sufficiently `small', $\rM_{\bp}(r, 0, \ba)$ can be constructed as an elementary GIT quotient. This section is devoted to the study of such small weight case.

\subsection{Moduli of parabolic bundles and GIT quotient}

Let $\pi_{i} : \Fl(V)^{n} \to \Fl(V)$ be the projection to the $i$-th factor. Then any line bundle on $\Fl(V)^{n}$ can be described as a restriction of $L_{\bfb} := \otimes \pi_{i}^{*}\cO(b_{\bullet}^{i})$. Note that there is a natural diagonal $\SL_{r}$-action on $\Fl(V)^{n}$. The GIT stability with respect to $L_{\bfb}$ is well-known:

\begin{theorem}[\protect{\cite[Theorem 11.1]{Dol03}}]\label{thm:GITstability}
A point $(W_{\bullet}^{i}) \in \Fl(V)^{n}$ is (semi)-stable with respect to $L_{\bfb}$ if and only if for every proper $s$-dimensional subspace $V' \subset V$, the following inequality holds:
\begin{equation}\label{eqn:GITstability}
	\frac{1}{s}\sum_{i=1}^{n}\sum_{j=1}^{r-1}b_{j}^{i}\dim (W_{j}^{i}
	\cap V') (\le) <
	\frac{1}{r}\left(\sum_{i=1}^{n}\sum_{j=1}^{r-1}jb_{j}^{i}\right).
\end{equation}
\end{theorem}

As in the case of parabolic stability, we define some terminology related to weight data. 

\begin{definition}
Let $L_{\bfb}$ be a linearization on $\Fl(V)^{n}$. 
\begin{enumerate}
\item $L_{\bfb}$ is \emph{effective} if $(\Fl(V)^{n})^{ss}(L_{\bfb}) \ne \emptyset$.
\item $L_{\bfb}$ is \emph{general} if $(\Fl(V)^{n})^{ss}(L_{\bfb}) = (\Fl(V)^{n})^{s}(L_{\bfb})$.
\end{enumerate}
\end{definition}

For a parabolic weight $\ba$, let $\bd = (d_{\bullet}^{1}, d_{\bullet}^{2}, \cdots, d_{\bullet}^{n})$ be a new collection of sequences defined by $d_{j}^{i} = a_{j}^{i} - a_{j+1}^{i}$. $\bd$ is called the associated \emph{difference data}. For $\bd$, let $|\bd|_{j} := \sum_{i=1}^{n}d_{j}^{i}$ and $|\bd| = \sum_{j=1}^{r-1}|\bd|_{j}$. Then $|\ba|_{j} = \sum_{k=j}^{r-1}|\bd|_{k}$ and $|\ba| = \sum_{j=1}^{r-1}j|\bd|_{j}$.

\begin{theorem}\label{thm:smallweightGIT}
Let $\ba$ be a parabolic weight and $\bd$ be the associated difference data. Suppose that $\ba$ is sufficiently small in the sense that
\begin{equation}\label{eqn:smallweight}
	\sum_{j=1}^{s}j(r-s)|\bd|_{j} + \sum_{j=s+1}^{r-1}s(r-j)|\bd|_{j} \le r
\end{equation}
for all $1 \le s \le r-1$. Then $\rM_{\bp}(r, 0, \ba) \cong \Fl(V)^{n}\git_{L_{\bd}}\SL_{r}$, where $L_{\bd} := \otimes \pi_{i}^{*}\cO(d_{\bullet}^{i})$.
\end{theorem}

\begin{proof}
It is straightforward to check that \eqref{eqn:smallweight} is equivalent to $r\sum_{j=1}^{s}|\ba|_{j} - s|\ba| \le r$.

Let $X = (\Fl(V)^{n})^{ss}(L_{\bd})$. Consider a family of parabolic bundles $\cE$ over $X$ by taking the trivial bundle, the restriction of the universal flag, and the parabolic weight $\ba$. We claim that this is a family of $\ba$-semistable parabolic bundles. Let $\cE_{x} = (E = V \otimes \cO, \{W_{\bullet}^{i}\}, \ba)$ be the fiber over $x \in X$.

Let $F$ be a rank $s$ subbundle of $E$ whose topological degree is negative. Let $\cF$ be the induced parabolic subbundle. Then $\mu(\cF) \le (-1 +\sum_{j=1}^{s}|\ba|_{j})/s \le |\ba|/r = \mu(\cE_{x})$. Thus $\cF$ is not a destabilizing subbundle. Since $E$ does not have a positive degree subbundle, the only possible destabilizing subbundle is of degree zero. This subbundle must be trivial because it cannot have any positive degree factor.

Suppose that $\cF = (V' \otimes \cO, \{{W'}_{\bullet}^{i}\},\bfb)$ is a rank $s$ parabolic subbundle induced by taking an $s$-dimensional subspace $V' \subset V$. Then
\[
\begin{split}
	\mu(\cF) &= \frac{1}{s}\sum_{i=1}^{n}\sum_{j=1}^{r}a_{j}^{i}
	\dim(W_{j}^{i} \cap F|_{p^{i}}/W_{j-1}^{i} \cap F|_{p^{i}})\\
	&= \frac{1}{s}\sum_{i=1}^{n}\sum_{j=1}^{r}a_{j}^{i}
	\dim(W_{j}^{i} \cap V'/W_{j-1}^{i} \cap V')
	= \frac{1}{s}\sum_{i=1}^{n}\sum_{j=1}^{r-1}d_{j}^{i}
	\dim (W_{j}^{i} \cap V')\\
	&\le
	\frac{1}{r}\left(\sum_{i=1}^{n}\sum_{j=1}^{r-1}jd_{j}^{i}\right) =
	\frac{1}{r}\left(\sum_{i=1}^{n}\sum_{j=1}^{r-1}a_{j}^{i}\right) =
	\mu(\cE_{x}).
\end{split}
\]
The inequality is obtained from Theorem \ref{thm:GITstability}. Therefore $\cE_{x}$ is semistable.

By the universal property, there is an $\SL_{r}$-invariant morphism $\pi : X \to \rM_{\bp}(r, 0, \ba)$. Therefore we have the induced map $\bar{\pi} : \Fl(V)^{n}\git_{L_{\bd}}\SL_{r} \to \rM_{\bp}(r, 0, \ba)$. It is straightforward to check that $\bar{\pi}$ is set-theoretically injective. Since it is an injective map between two normal varieties with the same dimension and the target space is irreducible, it is an isomorphism.
\end{proof}

\subsection{Picard group of GIT quotient}

For the later use, we compute the Picard group of the GIT quotient.

\begin{proposition}\label{prop:maxPicnumGIT}
Suppose that $n > 2r$. There is a general linearization $L_{\bfb} = \otimes \pi_{i}^{*}\cO(b_{\bullet}^{i})$ such that $\mathrm{Pic}(\Fl(V)^{n}\git_{L_{\bfb}}\SL_{r})$ is naturally identified with an index $r$ sublattice of $\Pic(\Fl(V)^{n}) \cong \ZZ^{(r-1)n}$. In particular, a general linearization which is very close to the symmetric linearization $L_{\ba}$ with $a_{j}^{i} \equiv 1$ has the property. 
\end{proposition}

\begin{proof}
First of all, we show that for the symmetric linearization $L_{\ba}$, the codimension of the non-stable locus $\Fl(V)^{n}\setminus (\Fl(V)^{n})^{ss}(L_{\ba})$ is at least two. If $(W_{\bullet}^{i})$ is not stable, then by \eqref{eqn:GITstability}, there is an $s$-dimensional subspace $V' \subset V$ such that
\begin{equation}\label{eqn:destabilizing}
	\frac{1}{s}\sum_{i=1}^{n}\sum_{j=1}^{r-1}
	\dim (W_{j}^{i}\cap V')\ge\frac{n(r-1)}{2}.
\end{equation}

Let $U_{s} \subset \Fl(V)^{n}$ be the set of flags $(W_{\bullet}^{i})$ which satisfy \eqref{eqn:destabilizing} for some $s$-dimensional subspace $V' \subset V$. Let $\widetilde{U}_{s} \subset \Gr(s, V) \times \Fl(V)^{n}$ be the space of pairs $(V', (W_{\bullet}^{i}))$ such that $(W_{\bullet}^{i})$ satisfies \eqref{eqn:destabilizing} for $V'$. There are two projections $p_{1}:\widetilde{U}_{s} \to \Gr(s, V)$ and $p_{2} : \widetilde{U}_{s} \to \Fl(V)^{n}$. The codimension of each fiber of $p_{1}$ in $\Fl(V)^{n}$ is $\lceil n(r-1)/2 - n(s-1)/2\rceil = \lceil n(r-s)/2\rceil$ because on \eqref{eqn:destabilizing}, the left hand side is $n(s-1)/2$ for a general $(W_{\bullet}^{i})$ and as the value of the left hand side increases by one, the codimension of the locus increases by one, too. Thus the codimension of $\widetilde{U}_{s}$ is at least $\lceil n(r-s)/2\rceil$. On the other hand, because $p_{2}(\widetilde{U}_{s}) = U_{s}$, the codimension of $U_{s}$ is at least $\lceil n(r-s)/2\rceil - s(r-s)$, which is at least two for any $1 \le s \le r-1$. The non-stable locus $\Fl(V)^{n} \setminus (\Fl(V)^{n})^{ss}(L_{\ba})$ is $\cup_{s = 1}^{r-1}U_{s}$, so it is of codimension at least two.

By perturbing the linearization slightly, we can obtain a general linearization $L_{\bfb}$. The unstable locus with respect to $L_{\bfb}$ is contained in the non-stable locus of $L_{\ba}$. Thus it is also of codimension at least two. In particular, $\Pic((\Fl(V)^{n})^{s}(L_{\bfb})) = \Pic(\Fl(V)^{n})$.

Let $\Pic^{\SL_{r}}((\Fl(V)^{n})^{s}(L_{\bfb})$ be the group of linearizations. Then there is an exact sequence
\[
	0 \to \Hom(\SL_{r}, \CC^{*}) \to
	\Pic^{\SL_{r}}((\Fl(V)^{n})^{s}(L_{\bfb})) \stackrel{\alpha}{\to}
	\Pic((\Fl(V)^{n})^{s}(L_{\bfb})) \to \Pic(\SL_{r})
\]
(\cite[Theorem 7.2]{Dol03}). Furthermore, $\Hom(\SL_{r}, \CC^{*}) = \Pic(\SL_{r}) = 0$. Thus $\alpha$ is an isomorphism.

By Kempf's descent lemma (\cite[Theorem 2.3]{DN89}), an $\SL_{r}$-linearized line bundle $L$ on $(\Fl(V)^{n})^{s}(L_{\bfb})$ descends to $\Fl(V)^{n}\git_{L_{\bfb}}\SL_{r}$ if and only if for every closed orbit $\SL_{r}\cdot x$, the stabilizer $\St_{x}$ acts on $L_{x}$ trivially. Lemma \ref{stabilizer} below tells us that for any $L \in \Pic(\Fl(V)^{n})$, $L^{r}$ descends to $\Fl(V)^{n}\git_{L_{\bfb}}\SL_{r}$.
\end{proof}

\begin{lemma}\label{stabilizer}
Let $x = (W_{\bullet}^{i})$ be a stable point on $\Fl(V)^{n}$ with respect to some linearization. Then $\St_{x}$ is isomorphic to the group of $r$-th root of unity.
\end{lemma}

\begin{proof}
Let $A \in \St_{x}$. $A$ has a finite order because $\St_{x}$ is finite. Since $\mathrm{char}\; \CC = 0$, the Jordan canonical form of $A$ cannot have any block of size larger than one. Thus we may assume that $A$ is a diagonal matrix. Decompose $V = \oplus_{\lambda}V_{\lambda}$ into eigenspaces with respect to $A$. Note that any invariant space with respect to $A$ has to be of the form $\oplus_{\lambda}W_{\lambda}$ where $W_{\lambda} \subset V_{\lambda}$. Thus all $W_{j}^{i}$ are of these forms. If we take a diagonal matrix $B$ which acts on $V_{\lambda}$ as a multiplication by $a^{\lambda}$ for some $a^{\lambda}$, then $B$ preserves all $W_{j}^{i}$, so $B \in \St_{x}$. But in this case $\dim \St_{x}$ is the number of distinct eigenvalues minus one. Since $x$ has a finite stabilizer, there is only one eigenvalue. Therefore $A$ is a scalar matrix.
\end{proof}

\begin{definition}\label{def:maxstlocus}
A general linearization $L_{\bfb}$ is called a linearization \emph{with a maximal stable locus} if $\Fl(V)^{n}\setminus (\Fl(V)^{n})^{ss}(L_{\bfb})$ is of codimension at least two, (so $\rho(\Fl(V)^{n}\git_{L_{\bfb}}\SL_{r}) = (r-1)n$). 
\end{definition}

\begin{remark}
For small $n$, Proposition \ref{prop:maxPicnumGIT} is not true. For instance, if $r = 2, n = 4$ or $r = n = 3$, for a general linearization, $\Fl(V)^{n}\git \SL_{r}$ is a unirational normal curve. Thus $\Fl(V)^{n}\git \SL_{r} \cong \PP^{1}$.
\end{remark}


\section{Wall crossing analysis}\label{sec:wallcrossing}

Here we describe how the moduli space is changed if one varies the parabolic weight.

\subsection{Walls and chambers}\label{ssec:wallchamber}

The space of all valid normalized parabolic weights is an open polytope
\[
	W_{r, n}^{o} := \{\ba = (a_{j}^{i})_{1 \le j \le r-1, 1 \le i \le n}\;|\;
	1 > a_{1}^{i} > a_{2}^{i} > \cdots > a_{r-1}^{i}>0\}
	\subset \RR^{(r-1)n}.
\]
Since the weight data is normalized, $a_{r}^{i} = 0$. Let $W_{r, n}$ be the closure of $W_{r, n}^{o}$.

The polytopes $W_{r,n}^{o}$ and $W_{r, n}$ have a natural wall-chamber structure: If two weights $\ba$ and $\ba'$ are on the same open chamber, then $\rM_{\bp}(r, 0, \ba) = \rM_{\bp}(r, 0, \ba')$. If $\ba$ is in one of open chambers, then $\ba$ is general thus $\rM_{\bp}(r, 0, \ba)$ is smooth. Note that it is possible that $\ba$ is not effective, so $\rM_{\bp}(r, 0, \ba) = \emptyset$. 

A parabolic weight $\ba$ is on a wall if there is a strictly semi-stable parabolic bundle $\cE = (E, \{W_{\bullet}^{i}\}, \ba)$. Then there is a unique destabilizing parabolic subbundle $\cF = (F, \{{W|_{F}}_{\bullet}^{i}\}, \bfb)$. For such an $\cF$,
\[
	\mu(\cF) = \frac{\deg F + \sum_{i=1}^{n}\sum_{j=1}^{r-1}a_{j}^{i}
	\dim ((W_{j}^{i} \cap F|_{p^{i}})/(W_{j-1}^{i} \cap F|_{p^{i}}))}
	{\rk F}
	= \frac{|\ba|}{r} = \mu(\cE).
\]
It occurs when there are two integers $d \le 0$ and $1 \le s \le r-1$, $n$ subsets $J^{i} \subset [r]$ of size $s$ such that
\[
	\frac{d + \sum_{i=1}^{n}\sum_{j \in J^{i}}a_{j}^{i}}{s}
	= \frac{|\ba|}{r}.
\]
Let $\cJ := \{J^{1}, J^{2}, \cdots, J^{n}\}$. Thus a stability wall is of the form
\[
	\Delta(s, d, \cJ) :=
	\{\ba \in W_{r, n}^{o}\;|\; r(d + \sum_{i=1}^{n}
	\sum_{j \in J^{i}}a_{j}^{i}) = s|\ba|\}.
\]
This equation is linear with respect to the variables $a_{j}^{i}$. So the polytope $W_{r, n}^{o}$ is divided by finitely many hyperplanes and each open chamber is a connected component of
\[
	W_{r, n}^{o} \setminus \left(\bigcup \Delta(s, d, \cJ)\right).
\]
For $\cJ = \{J^{1}, J^{2}, \cdots, J^{n}\}$, set $\cJ^{c} := \{[r]\setminus J^{1}, [r] \setminus J^{2}, \cdots, [r] \setminus J^{n}\}$. Then $\Delta(s, d, \cJ) = \Delta(r-s, -d, \cJ^{c})$. $\Delta(s, d, \{J, J, \cdots, J\})$ is denoted by $\Delta(s, d, nJ)$.

A wall-crossing is \emph{simple} if it is a wall-crossing along the relative interior of a wall. Because every wall-crossing can be decomposed into a finite sequence of simple wall-crossings, it is enough to study simple wall-crossings.

Fix a wall $\Delta(s, d, \cJ)$ and take a general point $\ba \in \Delta(s, d, \cJ)$. A small open neighborhood of $\ba$ is divided into two pieces by the wall. Let $\Delta(s, d, \cJ)^{+}$ and $\Delta(s, d, \cJ)^{-}$ be the two connected components such that
\[
	r(d + \sum_{i=1}^{n}\sum_{j \in J^{i}}a_{j}^{i}) > s|\ba|
\]
and
\[
	r(d + \sum_{i=1}^{n}\sum_{j \in J^{i}}a_{j}^{i}) < s|\ba|
\]
respectively. Let $\ba^{+}$ (resp. $\ba^{-}$) be a point on $\Delta(s, d, \cJ)^{+}$ (resp. $\Delta(s, d, \cJ)^{-}$).

There are two functorial morphisms (\cite[Theorem 3.1]{BH95}, \cite[Section 7]{Tha96})
\[
	\xymatrix{\rM_{\bp}(r, 0, \ba^{-}) \ar[rd]^{\phi^{-}} &&
	\rM_{\bp}(r, 0, \ba^{+}) \ar[ld]_{\phi^{+}}\\
	& \rM_{\bp}(r, 0, \ba).}
\]
Let $Y \subset \rM_{\bp}(r, 0, \ba)$ be the locus that one of $\phi^{\pm} : Y^{\pm} := {\phi^{\pm}}^{-1}(Y) \to Y$ is not an isomorphism. That means $\rM_{\bp}(r, 0, \ba^{-}) \setminus Y^{-} \cong \rM_{\bp}(r, 0, \ba) \setminus Y \cong \rM_{\bp}(r, 0, \ba^{+}) \setminus Y^{+}$. We call $Y^{\pm}$ as the \emph{wall-crossing center}. Suppose that $\cE = (E, \{W_{\bullet}^{i}\}, \ba)$ is on $Y$. Then there is a rank $s$ destabilizing subbundle $\cE^{+} = (E^{+}, \{{W|_{E^{+}}}_{\bullet}^{i}\}, \bfb)$ with $\mu(\cE^{+}) = \mu(\cE)$. We have a short exact sequence $0 \to \cE^{+} \to \cE \to \cE^{-} \to 0$ in the category of parabolic sheaves where $\cE^{-} = (E^{-} := E/E^{+}, \{{W/E^{+}}_{\bullet}^{i}\}, \bc)$. Then $Y$ parametrizes S-equivalent classes of $\cE^{+} \oplus \cE^{-}$. Therefore $Y \cong \rM_{\bp}(s, d, \bfb) \times \rM_{\bp}(r-s, -d, \bc)$.

For the same data, define $\bfb^{\pm}$ and $\bc^{\pm}$ by using $\ba^{\pm}$. Let $(\bE^{+},\bfb^{\pm})$ be the universal family on $\rM_{\bp}(s, d, \bfb^{\pm})$. Let $(\bE^{-},\bc^{\pm})$ be the universal family on $\rM_{\bp}(r-s, -d, \bc^{\pm})$. Let $\pi^{+}:\rM_{\bp}(s, d, \bfb^{+})\times\rM_{\bp}(r-s, -d, \bc^{+})\times\PP^{1}\to\rM_{\bp}(s, d, \bfb^{+})\times\rM_{\bp}(r-s, -d, \bc^{+})$ and $\pi^{-}:\rM_{\bp}(s, d, \bfb^{-})\times\rM_{\bp}(r-s, -d, \bc^{-})\times\PP^{1}\to\rM_{\bp}(s, d, \bfb^{-})\times\rM_{\bp}(r-s, -d, \bc^{-})$ be projections. The exceptional fibers of $\phi^{-}$ and $\phi^{+}$ are projective bundles
\[
	Y^{-} = \PP R^{1}\pi^{-}_{*}
	\cpHom((\bE^{-},\bc^{-}),(\bE^{+},\bfb^{-}))
\]
and
\[
	Y^{+} = \PP R^{1}\pi^{+}_{*}
	\cpHom((\bE^{+},\bfb^{+}),(\bE^{-},\bc^{+}))
\]
respectively. Fiberwisely, ${\phi^{-\;-1}}(\cE^{+} \oplus \cE^{-}) = \PP \Ext^{1}((E^{-}, \{{W/E^{+}}_{\bullet}^{i}\}, \bc^{-}), (E^{+}, \{{W|_{E^{+}}}_{\bullet}^{i}\}, \bfb^{-}))$ and ${\phi^{+ -1}}(\cE^{+} \oplus \cE^{-}) = \PP \Ext^{1}((E^{+}, \{{W|_{E^{+}}}_{\bullet}^{i}\}, \bfb^{+}), (E^{-}, \{{W/E^{+}}_{\bullet}^{i}\}, \bc^{+}))$.

\begin{remark}\label{rem:wallcrossing}
Furthermore, we can describe the change of the fibers in detail. An element $\cE \in Y^{-}$ fits into a short exact sequence $0 \to \cE^{+} \to \cE \to \cE^{-} \to 0$ in the category of parabolic bundles. In particular, their underlying bundles fits into $0 \to E^{+} \to E \to E^{-} \to 0$. After the wall-crossing, $\phi^{+ -1}(\phi^{-}([\cE]))$ is the set of parabolic bundles $\cF$ which fits into the sequence $0 \to \cE^{-} \to \cF \to \cE^{+} \to 0$. Again, its underlying bundle fits into the exact sequence $0 \to E^{-} \to F \to E^{+} \to 0$.
\end{remark}

\begin{proposition}[\protect{\cite[Section 7]{Tha96}}]\label{prop:wallcrossing}
Suppose that $\rM_{\bp}(r, 0, \ba^{\pm})$ are nonempty. The blow-up of $\rM_{\bp}(r, 0, \ba^{-})$ along $Y^{-}$ is isomorphic to the blow-up of $\rM_{\bp}(r, 0, \ba^{+})$ along $Y^{+}$. In paticular, $\dim Y^{-}+\dim Y^{+}-\dim Y=\dim\rM_{\bp}(r, 0, \ba)-1$.
\end{proposition}

Note that for some weight data $\ba^{+}$, the moduli space $\rM_{\bp}(r, 0, \ba^{+})$ may be empty.

\begin{proposition}\label{prop:boundarywall}
Suppose that $\rM_{\bp}(r, 0, \ba^{+}) = \emptyset$. Then $\rM_{\bp}(r, 0, \ba^{-})$ has a projective bundle structure over $\rM_{\bp}(r, 0, \ba) = \rM_{\bp}(s, d, \bfb^{-}) \times \rM_{\bp}(r-s, -d, \bc^{-})$.
\end{proposition}

\begin{proof}
Note that $\rM_{\bp}(r, 0, \ba^{+}) = \emptyset$ only if $Y^{-} = \rM_{\bp}(r, 0, \ba^{-})$. Conversely, if $Y^{-} = \rM_{\bp}(r, 0, \ba^{-})$ and $\rM_{\bp}(r, 0, \ba^{+}) \ne \emptyset$, it is straightforward to make a contradiciton from Proposition \ref{prop:wallcrossing} since $\dim Y^{+} \ge \dim Y$ and $\dim Y^{-} \ge \dim \rM_{\bp}(r, 0, \ba)$.
\end{proof}

\subsection{Scaling up}

In this section, we examine a special kind of wall-crossing. Let $\ba$ be a general parabolic weight. For a positive real number $c > 0$, define a parabolic weight $\ba(c)$ as $a(c)_{j}^{i} := ca_{j}^{i}$. When $c = \epsilon \ll 1$, $\ba(c)$ satisfies the smallness condition in Theorem \ref{thm:smallweightGIT}, so $\rM_{\bp}(r, 0, \ba(\epsilon)) = \Fl(V)^{n}\git_{L_{\ba(\epsilon)}}\SL_{r}$. As $c$ increases, we may cross several walls. By perturbing if it is necessary, we may assume that all wall-crossings are simple. We will call this type of wall-crossings as a \emph{scaling wall-crossing}.

Suppose that $\Delta(s, d, \cJ)$ is a wall we can meet and $\ba^{0} := \ba(c) \in \Delta(s, d, \cJ)$. Let $\ba^{\pm} = \ba(c\pm \epsilon)$. For a parabolic bundle $\cE = (E, \{W_{\bullet}^{i}\}, \ba^{0}) \in Y$, let $\cE^{+}$ be the destabilizing subbundle with respect to $\ba^{0}$ and $\cE^{-}$ be the quotient $\cE/\cE^{+}$. The induced weight data of $\cE^{+}$ with respect to $\ba^{\pm}$ is denoted by $\bfb^{\pm}$, as before.

Here we would like to compute $\dim \Ext^{1}(\cE^{-}, \cE^{+})$ with respect to $\ba^{-}$. By Serre duality (Lemma \ref{lem:Serreduality}), $\Ext^{1}(\cE^{-}, \cE^{+}) \cong \spHom(\cE^{+}\otimes \cO(-(n-2)),\cE^{-})^{*}$.
Consider the following short exact sequence of sheaves
\begin{equation}\label{eqn:exseq}
\begin{split}
0 &\to \cspHom(\cE^{+}\otimes \cO(-(n-2)), \cE^{-}) \to
\cHom(E^{+} \otimes \cO(-(n-2)), E^{-})\\
&\to \frac{\bigoplus_{i=1}^{n}\Hom(E^{+}\otimes \cO(-(n-2))|_{p^{i}},
E^{-}|_{p^{i}})}{\bigoplus_{i=1}^{n}N_{p^{i}}(\cE^{+}\otimes \cO(-(n-2)), \cE^{-})} \to 0
\end{split}
\end{equation}
where $N_{p}(\cE_{1}, \cE_{2})$ is the subspace of strictly parabolic maps in $\Hom(E_{1}|_{p}, E_{2}|_{p})$ at $p \in \PP^{1}$. For $\ba^{-}$, $\mu(\cE^{-}) > \mu(\cE^{+})$. Because $(\cE^{-}, \bfb^{-})$ and $(\cE^{+}, \bc^{-})$ are stable,
\[
	\rH^{1}(\cspHom(\cE^{+}\otimes \cO(-(n-2)), \cE^{-}))
	= \Ext^{0}(\cE^{-}, \cE^{+})^{*} = \pHom(\cE^{-}, \cE^{+})^{*}
	= 0
\]
by Lemma \ref{lem:Serreduality}. Thus we have an exact sequence of vector spaces
\begin{equation}\label{eqn:exseqglobal}
\begin{split}
0 &\to \spHom(\cE^{+}\otimes \cO(-(n-2)), \cE^{-}) \to
\Hom(E^{+} \otimes \cO(-(n-2)), E^{-})\\
&\to \frac{\bigoplus_{i=1}^{n}\Hom(E^{+}\otimes \cO(-(n-2))|_{p^{i}},
E^{-}|_{p^{i}})}{\bigoplus_{i=1}^{n}N_{p^{i}}(\cE^{+}\otimes \cO(-(n-2)), \cE^{-})} \to 0.
\end{split}
\end{equation}

Recall that at each parabolic point $p^{i}$, the intersection of $E^{+}$ with $E|_{p^{i}}$ is described by an $s$-subset $J^{i} \subset [r]$.

\begin{lemma}\label{lem:dimN}
Suppose that $\rk E^{+} = s$. In the above situation,
\[
	\dim N_{p^{i}}(\cE^{+}\otimes \cO(-(n-2)), \cE^{-})
	= \dim \omega_{J^{i}}
\]
where $\omega_{J^{i}}$ is the Schubert class in $\rH^{*}(\Gr(s, r))$ associated to the increasing sequence $J^{i}$.
\end{lemma}

\begin{proof}
Note that $\dim N_{p^{i}}(\cE^{+}\otimes \cO(-(n-2)), \cE^{-}) = \dim N_{p^{i}}(\cE^{+}, \cE^{-})$. At the fiber of $p^{i}$, take an ordered basis $\{e_{j}\}$ of $E|_{p^{i}}$ by choosing a nonzero vector $e_{j}$ for each $W_{j}^{i}\setminus W_{j-1}^{i}$. Then $E^{+}|_{p^{i}}$ (resp. $E^{-}|_{p^{i}}$) is spanned by $\{e_{j}\}_{j \in J^{i}}$ (resp. $\{e_{j}\}_{j \in [r] \setminus J^{i}}$). Now to construct a map in $N_{p^{i}}(\cE^{+}, \cE^{-})$, $e_{j} \in E^{+}|_{p^{i}}$ can be mapped into the subspace of generated by $e_{k}$ where $k < j$ and $k \notin J^{i}$. Therefore the dimension is $\sum_{j=1}^{s}(J_{j}^{i} - j)$ and this is equal to $\dim \omega_{J^{i}}$.
\end{proof}

Now from \eqref{eqn:exseqglobal} and Lemma \ref{lem:dimN}, we obtain the following result.

\begin{proposition}\label{prop:extdim}
Suppose that $n$ is sufficiently large. Let $\cE$ be a parabolic bundle on $Y \subset \rM_{\bp}(r, 0, \ba)$ for a scaling wall $\Delta(s, d, \cJ)$. Let $\cE^{+}$ be the destabilizing subbundle and $\cE^{-} = \cE/\cE^{+}$. With respect to the parabolic weight $\ba^{-}$,
\begin{equation}\label{eqn:dimExt}
	\dim \Ext^{1}(\cE^{-}, \cE^{+})
	= \dim \Hom(E^{+}\otimes \cO(-(n-2)), E^{-})
	- ns(r-s) + \sum_{i=1}^{n}\dim \omega_{J^{i}}.
\end{equation}
\end{proposition}

Let $\ba$ be a general parabolic weight and consider the scaling wall-crossing. For the weight data $\ba(\epsilon)$, $\rM_{\bp}(r, 0, \ba(\epsilon))$ is the GIT quotient $\Fl(V)^{n}\git_{L_{\ba(\epsilon)}}\SL_{r}$, and any underlying vector bundle $E$ of $\cE \in \rM_{\bp}(r, 0, \ba(\epsilon))$ is trivial.

The first wall we can meet while the scaling wall-crossing is of the form $\Delta(s, -1, n[s])$. Let $\Delta(s, d, \cJ)$ be the first wall. Because $E$ does not have any positive degree subbundle, $d \le 0$. A wall of the form $\Delta(s, 0, \cJ)$ does not appear while scaling. (These walls are GIT walls.) The maximal parabolic slope we can obtain occurs when $|d|$ is the smallest one and $J^{i} = [s]$. Indeed only two of them actually occur.

\begin{lemma}
The first wall is either $\Delta(1, -1, n[1])$ or $\Delta(r-1, -1, n[r-1])$. Moreover, only one of them occurs during the scaling wall-crossing.
\end{lemma}

\begin{proof}
We show that the walls $\Delta(s, -1, n[s])$ for $2 \le s \le r-2$ do not intersect $W_{r, n}^{o}$.

Suppose that $\Delta(s, -1, n[s]) \cap W_{r,n}^{o}$ is nonempty and the wall actually provides a nontrivial wall-crossing for $2 \le s \le r-2$. For notational simplicity, we may assume that $\ba = \ba(1) \in \Delta(s, -1, n[s])$. $\cE = (E = \cO^{r}, \{W_{\bullet}^{i}\}, \ba^{-}) \in Y^{-}$ has a parabolic subbundle $\cF = (F \cong \cO(-1)\oplus \cO^{s-1}, \{{W|_{F}}_{\bullet}^{i}\}, \bfb^{-})$ such that ${W|_{F}}_{j}^{i} = W_{j}^{i}$ for all $i$ and $1 \le j \le s$, and $\mu(\cF) = \mu(\cE)$ with respect to $\ba$. Then there is an $(s+1)$-dimensional vector space $V'$ such that $F \to E$ factors through $F \to V' \otimes \cO \to E$. Let $\cF'$ be the induced parabolic subbundle of $\cE$ whose underlying bundle is $V' \otimes \cO$. With respect to $\ba^{-} = \ba(1-\epsilon)$, $\cE$ is stable. Thus we have
\[
	\frac{1}{s+1}\sum_{i=1}^{n}\sum_{j=1}^{s}a_{j}^{i}
	= \mu(\cF') < \mu(\cE) = \frac{1}{r}|\ba|.
\]
On the other hand, let $F'' \subset F$ be any rank one trivial subbundle and let $\cF''$ be the parabolic subbundle induced by $F''$. Then we have
\[
	\sum_{i=1}^{n}a_{s}^{i} \le \mu(\cF'') < \mu(\cE) = \frac{1}{r}|\ba|.
\]
By taking the weighted average of left sides, we have
\[
	\frac{1}{r}|\ba| >
	\frac{1}{r}\left((s+1)\frac{1}{s+1}
	\sum_{i=1}^{n}\sum_{j=1}^{s}a_{j}^{i}
	+ (r-s-1)\sum_{i=1}^{n}a_{s}^{i}\right)
	= \frac{1}{r}\sum_{i=1}^{n}\left(\sum_{j=1}^{s-1}a_{j}^{i}
	+ (r-s)a_{s}^{i}\right) 	> \frac{1}{r}|\ba|
\]
and this is a contradiction.

Now suppose that the first wall is $\Delta(1, -1, n[1])$. We may assume that $\ba \in \Delta(1, -1, n[1])$. $\cE = (E= \cO^{r}, \{W_{\bullet}^{i}\}, \ba^{-}) \in Y^{-}$ has a subbundle $\cF$ whose underlying bundle $F$ is $\cO(-1)$ such that $F|_{p^{i}} = W_{1}^{i}$. We can take a 2-dimensional $V'$ and $\cF'$ as before. Then
\[
	\frac{1}{2}\sum_{i=1}^{n}a_{1}^{i}
	\le \mu(\cF') < \mu(\cE) = \frac{1}{r}|\ba|.
\]

Now assume that $\cE$ is also in the wall-crossing center for $\Delta(r-1, -1, n[r-1])$. Then there is a subbundle $G \cong \cO(-1)\oplus \cO^{r-2}$ of $E$ such that $G|_{p^{i}} = W_{r-1}^{i}$. Let $\cG$ be the induced parabolic subbundle whose underlying bundle is $G$. Let $V'$ be the $(r-2)$-dimensional vector space such that $G = \cO(-1) \oplus (V' \otimes \cO)$ and let $\cG'$ be the parabolic subbundle associated to $V' \otimes \cO$. Then
\[
	\frac{1}{r-2}\sum_{i=1}^{n}\sum_{j=2}^{r-1}a_{j}^{i}
	\le \mu(\cG') < \mu(\cE)
	= \frac{1}{r}|\ba|.
\]
An weighted average of the left hand sides of is $\mu(\cE)$. This makes a contradiction.
\end{proof}

\begin{remark}
The proof tells us that two wall-crossing centers $Y_{1}^{-}$ for $\Delta(1, -1, n[1])$ and $Y_{r-1}^{-}$ for $\Delta(r-1, -1, n[r-1])$ cannot be simultaneously stable on $\Fl(V)^{n}\git_{L} \SL_{r}$ for any linearization $L$.
\end{remark}

The first wall-crossing, which is either along $\Delta(1, -1, n[1])$ or $\Delta(r-1, -1, n[r-1])$ depending on $\ba$, is always a blow-up.

\begin{lemma}\label{lem:firstwallcrossing}
The scaling wall-crossing along $\Delta(1, -1, n[1])$ (resp. $\Delta(r-1, -1, n[r-1])$) is a blow-up along $Y^{-} \cong \rM_{\bp}(r-1, 1, \bc)$ (resp. $\rM_{\bp}(r-1, -1, \bfb)$).
\end{lemma}

\begin{proof}
By Proposition \ref{prop:wallcrossing}, the wall-crossing is a blow-up if and only if $Y^{-} = Y$ if and only if $\dim \Ext^{1}(\cE^{-}, \cE^{+})$ with respect to $\ba^{-}$ is one for $\cE \in Y^{-}$.

Consider the first case of $\Delta(1, -1, n[1])$. The underlying bundle $E^{+}$ is $\cO(-1)$ and $E^{-} = \cO(1)\oplus \cO^{r-2}$. Thus $\Hom(E^{+}\otimes \cO(-(n-2)),  E^{-}) \cong \rH^{0}(\cO(n) \oplus \cO(n-1)^{r-2})$. Because $J^{i} = [1]$, $\dim \omega_{J^{i}} = 0$. By using Proposition \ref{prop:extdim}, it is straightforward to see that $\dim \Ext^{1}(\cE^{-}, \cE^{+}) = 1$. The blow-up center is obtained from the description in Section \ref{ssec:wallchamber} and the fact that $\rM_{\bp}(1, d, \bfb)$ is a point. The other case is similar.
\end{proof}

Now suppose that $\ba$ is a general small weight such that $\rM_{\bp}(r, 0, \ba) = \Fl(V)^{n}\git_{L_{\ba}} \SL_{r}$ and $\rho(\Fl(V)^{n}\git_{L_{\ba}}\SL_{r}) = (r-1)n$. By scaling up the weight, for some $c > 1$, the weight data $\ba(c)$ crosses either $\Delta(1, -1, n[1])$-wall or $\Delta(r-1, -1, n[r-1])$-wall. Then for $\ba(c+\epsilon)$, $\rM_{\bp}(r, 0, \ba(c+\epsilon))$ has Picard number $(r-1)n + 1$, which is maximal because the moduli stack $\rM_{\bp}(r, 0)$ has the same Picard number. In particular, $\mathrm{Cox}(\rM_{\bp}(r, 0))$ is identified with $\mathrm{Cox}(\rM_{\bp}(r, 0, \ba(c+\epsilon)))$.

\begin{definition}\label{def:dominantweight}
A general parabolic weight $\ba$ is \emph{dominant} if $\rho(\rM_{\bp}(r, 0, \ba)) = (r-1)n + 1$.
\end{definition}

Such a weight is called dominant because any other $\rM_{\bp}(r, 0, \bfb)$ can be obtained as a rational contraction of $\rM_{\bp}(r, 0, \ba)$ (See Section \ref{sec:MMP}). 

The exceptional divisor $Y^{+}$ parametrizes the parabolic bundles with nontrivial underlying bundles (Remark \ref{rem:wallcrossing}). Thus $Y^{+}$ is the generalized theta divisor $\Theta$ described in Section \ref{ssec:conformalblock}. As a conformal block, it is identified with $\VV_{1, \vec{0}}^{\dagger}$.

The Cox ring of the moduli stack $\rM_{\bp}(r, 0)$ can be identified with that of a projective moduli space. 

\begin{proposition}\label{prop:coxringdominantweight}
Let $\ba$ be a dominant weight. Then $\mathrm{Cox}(\rM_{\bp}(r, 0, \ba)) = \mathrm{Cox}(\rM_{\bp}(r, 0))$. 
\end{proposition}

\begin{proof}
From the inclusion $\rM_{\bp}(r, 0, \ba) \hookrightarrow \rM_{\bp}(r, 0)$, we have a morphism of Picard groups $\mathrm{Pic}(\rM_{\bp}(r, 0)) \to \mathrm{Pic}(\rM_{\bp}(r, 0, \ba))$ and that of Cox rings $r : \mathrm{Cox}(\rM_{\bp}(r, 0)) = \VV^{\dagger} \to \mathrm{Cox}(\rM_{\bp}(r, 0, \ba))$. We claim that $r$ is an isomorphism. 

We may assume that $\rM_{\bp}(r, 0, \ba)$ is the blow-up of $\Fl(V)^{n}\git_{L}\SL_{r}$ where $L$ is a linearization with a maximal stable locus. Let $\pi : \rM_{\bp}(r, 0, \ba) \to \Fl(V)^{n}\git_{L}\SL_{r}$ be the blow-up morphism. The exceptional divisor is $\Theta$. 

Suppose that $s \in \VV_{\ell, \vec{\lambda}}^{\dagger}$. For some $\ell' \gg 0$ determined by $\vec{\lambda}$, $\VV_{\ell', \vec{\lambda}}^{\dagger} \cong V_{\vec{\lambda}}^{\SL_{r}}$, the space of classical invariant factors. Note that $V_{\vec{\lambda}}^{\SL_{r}}$ is the space of global sections of the descent $\overline{F}_{\vec{\lambda}}$ on $\Fl(V)^{n}\git_{L}\SL_{r}$ of $F_{\vec{\lambda}} := \otimes_{i=1}^{n}F_{\lambda^{i}}$. If $\ell \ge \ell'$, so $\VV_{\ell, \vec{\lambda}}^{\dagger} \cong \VV_{\ell', \vec{\lambda}}^{\dagger} = V_{\vec{\lambda}}^{\SL_{r}}$, then $s$ can be regarded as a section of $\rH^{0}(\rM_{\bp}(r, 0, \ba), \pi^{*}\overline{F}_{\vec{\lambda}} + (\ell - \ell')\Theta) \cong \rH^{0}(\rM_{\bp}(r, 0, \ba), \pi^{*}\overline{F}_{\vec{\lambda}})$. If $\ell < \ell'$, $s$ can be regarded as an invariant in $V_{\vec{\lambda}}^{\SL_{r}}$ which vanishes along $\Theta$ with multiplicity $\ell' - \ell$. Thus it is a section of $\rH^{0}(\rM_{\bp}(r, 0, \ba), \pi^{*}\overline{F}_{\lambda} - (\ell' - \ell)\Theta)$. 

Since any line bundle on $\rM_{\bp}(r, 0, \ba)$ with a nonzero global section can be written uniquely as $\pi^{*}\overline{F}_{\lambda} - m\Theta$ for some $\vec{\lambda}$ and $m \in \ZZ$, we obtain the desired result.
\end{proof}


\section{The moduli space is a Mori dream space}\label{sec:MDS}

In this section, we prove Theorem \ref{thm:Fanotype}. The finite generation of $\VV^{\dagger}$ (Theorem \ref{thm:fingen}) follows immediately.

\begin{theorem}\label{thm:Fanotype}
For any rank $r$ and a general parabolic weight $\ba$, $\rM_{\bp}(r, 0, \ba)$ is of Fano type.
\end{theorem}

Recall that a $\QQ$-factorial normal varieity $X$ is of \emph{Fano type} if there is an effective $\QQ$-divisor $\Delta$ such that $-(K_{X}+\Delta)$ is ample and $(X, \Delta)$ is a klt pair.

By \cite[Corollary 1.3.2]{BCHM10}, a $\QQ$-factorial normal variety of Fano type is a Mori dream space.

\begin{corollary}\label{cor:MDS}
For any general parabolic weight $\ba$, $\rM_{\bp}(r, 0, \ba)$ is a Mori dream space.
\end{corollary}

\begin{theorem}\label{thm:fingen}
The algebra $\VV^{\dagger}$ of conformal blocks is finitely generated.
\end{theorem}

\begin{proof}
By Proposition \ref{prop:coxringdominantweight}, the Cox ring of the moduli stack $\rM_{\bp}(r, 0)$ is the same with that of $\rM_{\ba}(r, 0, \ba)$ if $\ba$ is dominant. If $n > 2r$, by Theorem \ref{thm:smallweightGIT} and Proposition \ref{prop:maxPicnumGIT}, there is a small weight $\bfb$ such that $\rho(\rM_{\bp}(r, 0, \bfb)) = (r-1)n$. During the scaling wall-crossing, the first wall-crossing is a blow-up. Thus there is a dominant weight $\ba = \bfb(c)$. Then $\mathrm{Cox}(\rM_{\bp}(r, 0)) = \mathrm{Cox}(\rM_{\bp}(r, 0, \ba))$ is finitely generated by Corollary \ref{cor:MDS} and \cite[Theorem 2.9]{HK00}.

When $n$ is small, by Lemma \ref{lem:extension}, for a sufficiently large point configuration $\bq \supset \bp$, there is a morphism $\rM_{\bq}(r, 0, \ba') \to \rM_{\bp}(r, 0, \ba)$ for some $\ba'$. Then $\Pic(\rM_{\bp}(r, 0, \ba))$ is a subgroup of $\Pic(\rM_{\bq}(r, 0, \ba'))$. Let $H := \Hom(\Pic(\rM_{\bq}(r, 0, \ba'))/\Pic(\rM_{\bp}(r, 0, \ba)), \CC^{*})$. There is a natural action of $H$ on $\mathrm{Cox}(\rM_{\bq}(r, 0, \ba'))$ and by propagation of vacua (\cite[Theorem 3.15]{Uen08}), 
\[
	\mathrm{Cox}(\rM_{\bp}(r, 0, \ba)) \cong 
	\bigoplus_{\ell, \vec{\lambda}}\VV_{\ell, (\lambda^{1}, \lambda^{2}, 
	\cdots, \lambda^{n})}^{\dagger} \cong 
	\bigoplus_{\ell, \vec{\lambda}}
	\VV_{\ell, (\lambda^{1}, \lambda^{2}, \cdots, \lambda^{n}, 
	0, 0, \cdots, 0)}^{\dagger} \cong
	\mathrm{Cox}(\rM_{\bq}(r, 0, \ba'))^{H}.
\]
A torus-invariant subring of a finitely generated algebra is finitely generated, too (\cite[Theorem 3.3]{Dol03}).
\end{proof}

\begin{lemma}\label{lem:extension}
Let $\ba$ be a general effective parabolic weight. Then for any finite point configuration $\bq \supset \bp$, there is a parabolic weight $\ba'$ such that there is a morphism $\rM_{\bq}(r, 0, \ba') \to \rM_{\bp}(r, 0, \ba)$. 
\end{lemma}

\begin{proof}
Suppose that $\bp = (p^{1}, p^{2}, \cdots, p^{n})$ and $\bq = (p^{1}, p^{2}, \cdots, p^{n+m})$. Let $\ba'$ be a parabolic weight such that ${a'}_{\bullet}^{i} = a_{\bullet}^{i}$ for $i \le n$ and ${a'}_{j}^{i}$ are sufficiently small for $i > n$. There is a natural `forgetful' map
\begin{eqnarray*}
	\rM_{\bp}(r, 0, \ba') & \to & \rM_{\bq}(r, 0, \ba)\\
	(E, \{W_{\bullet}^{i}\}, \ba') & \mapsto &
	(E, \{W_{\bullet}^{i}\}_{i \le n}, \ba).
\end{eqnarray*}
This map is regular, because small weights $(a_{\bullet}^{i})_{i > n}$ do not affect on the inequalities for stability. 
\end{proof}

\begin{remark}
The proof of Theorem \ref{thm:fingen} does not provide any explicit set of generators. When  $r \le 3$ and $\bp$ is a generic configuration of points, by using a degeneration method, Manon showed that the set of $r^{n-1}$ level one conformal blocks generates $\VV^{\dagger}$ (\cite[Theorem 1.5]{Man09b}, \cite[Theorem 3]{Man13}). For $r \ge 4$, the set of level one conformal blocks is insufficient to generate $\VV^{\dagger}$. We expect that the generic configuration assumption is not essential.
\end{remark}

The remaining part of this section is devoted to the proof of Theorem \ref{thm:Fanotype}. We start with the computation of the canonical divisor. $\overline{\cO}(D)$ denotes the descent of a line bundle $\cO(D)$ on $X$ to the GIT quotient $X\git G$.

\begin{lemma}\label{lem:KGIT}
Let $L$ be a general linearization on $\Fl(V)^{n}$ with a maximal stable locus. The canonical divisor $K$ of $\Fl(V)^{n}\git_{L} \SL_{r}$ is
\[
	\otimes_{i=1}^{n}\pi_{i}^{*}\overline{\cO}(-2,-2,\cdots,-2).
\]
\end{lemma}

\begin{proof}
Because the canonical divisor is $S_{n}$-invariant, we have $K = \otimes_{i=1}^{n}\pi_{i}^{*}\overline{\cO}(b_{1},\cdots,b_{r-1})$. Let $\widetilde{C}_{j}^{i} \cong \PP^{1} \subset \Fl(V)^{n}$ be a Schubert curve that is obtained by taking a family of parabolic bundles $(\cO^{r}, \{W_{\bullet}^{i}\})$ such that $W_{\bullet}^{l}$ for $l\ne i$ and $W_{k}^{i}$ for $k\ne j$ are fixed and general, but $W_{j}^{i}$ is varying as a one-dimensional family of subspaces in $W_{j+1}^{i}$ containing $W_{j-1}^{i}$.Since $L$ is a linearization with a maximal stable locus, $\widetilde{C}_{j}^{i}$ does not intersect the unstable locus whose codimension is at least two. Thus by taking its image in $\Fl(V)^{n}\git_{L}\SL_{r}$, we have a curve $C_{j}^{i}$. By projection formula, $C_{j}^{i}\cdot \otimes_{i=1}^{n}\pi_{i}^{*}\overline{\cO}(b_{1},\cdots,b_{r-1}) = b_{j}$.

Note that $\widetilde{C}_{j}^{i}$ is indeed a curve on a fiber of the projection map $p : \Fl(V)^{n} \to \Fl(V)^{n-1}$ which forgets the $i$-th factor. The tangent bundle of the $i$-th factor $\Fl(V)$ is the quotient
\[
	0 \to \Hom_{F_{\cdot}}(V, V)\otimes \cO \to \Hom(V, V) \otimes \cO
	\to \cT_{\Fl(V)} \to 0
\]
where $\Hom_{F\cdot}(V, V)$ is the space of endomorphisms which preserve the flag. The restriction of the sequence to $\widetilde{C}_{j}^{i}$ is isomorphic to
\[
	0 \to \cO^{(r^{2}+r-4)/2}\oplus \cO(-1)^{2} \to \cO^{r^2} \to
	\cT_{\Fl(V)}|_{\widetilde{C}_{j}^{i}} \to 0.
\]
So $\deg \cT_{\Fl(V)}|_{\widetilde{C}_{j}^{i}} = 2$, and thus $\deg \cT_{\Fl(V)^{n}}|_{\widetilde{C}_{j}^{i}} = 2$. By the $\SL_{r}$-action, $\widetilde{C}_{j}^{i}$ deforms without fixed points. Thus along the fiber of the quotient map, the restriction of the tangent bundle is trivial. Therefore $\deg \cT_{\Fl(V)^{n}\git_{L}\SL_{r}}|_{C_{j}^{i}} = 2$. So $b_{j} = C_{j}^{i}\cdot K = -2$.
\end{proof}

\begin{corollary}\label{cor:Kconformal}
For a general small weight $\ba$, $\rH^{0}(-K) = \VV_{(r-1)n,(\lambda, \lambda, \cdots, \lambda)}^{\dagger}$ where $\lambda = 2(\sum_{j=1}^{r-1}\omega_{j})$.
\end{corollary}

\begin{proof}
By Lemma \ref{lem:KGIT}, $-K = \otimes_{i=1}^{n}\pi_{i}^{*}\overline{\cO}(2, 2, \cdots, 2)$. This is a product of $(r-1)n$ line bundles of the form $\pi_{i}^{*}\overline{\cO}(e_{a}) \otimes \pi_{j}^{*}\overline{\cO}(e_{r-a})$ where $e_{k}$ is the standard $k$-th vector. Each $\rH^{0}(\pi_{i}^{*}\overline{\cO}(e_{a}) \otimes \pi_{j}^{*}\overline{\cO}(e_{r-a}))$ is identified with $\VV_{1,(0, \cdots, 0, \omega_{r-a}, 0, \cdots, 0, \omega_{a}, 0, \cdots)}^{\dagger}$ where $\omega_{r-a}$ is on the $i$-factor and $\omega_{a}$ is on the $j$-th factor (\cite[Proposition 1.3]{BGM15}). Now by taking the tensor product of them, we obtain the statement.
\end{proof}

\begin{proposition}\label{prop:canonicaldivsior}
Let $\ba$ be a dominant parabolic weight. Let $\rM = \rM_{\bp}(r, 0, \ba)$. Then $\rH^{0}(-K_{\rM}) = \VV_{2r,(\lambda, \lambda, \cdots, \lambda)}^{\dagger}$ where $\lambda = (2\sum_{j=1}^{r-1}\omega_{j})$. 
\end{proposition}

\begin{proof}
We may assume that $\ba$ is the weight data right after the first wall-crossing while the scaling wall-crossing from $\ba(\epsilon)$. By Lemma \ref{lem:firstwallcrossing}, the first wall-crossing is the blow-up along $\rM_{\bp}(r-1, 1, \bc)$ or $\rM_{\bp}(r-1, -1, \bfb)$. In particular, the codimension of the blow-up center is $(r-1)n - 2r + 1$. By the blow-up formula of canonical divisors, if $K$ denotes the canonical divisor of $\Fl(V)^{n}\git_{L_{\ba}}\SL_{r}$ and if $\pi : \rM \to \Fl(V)^{n}\git_{L_{\ba}}\SL_{r}$ is the blow-up morphism, $-K_{\rM} = \pi^{*}(-K) - ((r-1)n - 2r)Y^{+}$. Since $Y^{+}$ is the theta divisor $\Theta$, $\rH^{0}(-K_{\rM}) = \VV_{(r-1)n - ((r-1)n - 2r),(\lambda, \lambda, \cdots, \lambda)}^{\dagger} = \VV_{2r,(\lambda, \lambda, \cdots, \lambda)}^{\dagger}$.
\end{proof}

A key theorem is the following classical result of Pauly.

\begin{theorem}[\protect{\cite[Theorem 3.3, Corollary 6.7]{Pau96}}]\label{thm:Pauly}
Let $\ba = (a_{\bullet}^{i})$ be a parabolic weight. Then there is an ample line bundle $\Theta_{\ba}$ on $\rM_{\bp}(r, 0, \ba)$ such that $\rH^{0}(\Theta_{\ba})$ is canonically identified with $\VV_{\ell,(\lambda^{1}, \lambda^{2}, \cdots, \lambda^{n})}^{\dagger}$ where $\ell$ is the smallest positive integer such that $\ell a_{j}^{i} \in \ZZ$ and $\lambda_{j}^{i} = \ell a_{j}^{i}$.
\end{theorem}

Let $\ba_{c}$ be a parabolic weight such that $({a_{c}}_{\bullet}^{i}) = \frac{1}{r}(r-1, r-2, \cdots, 1)$. By Theorem \ref{thm:Pauly}, $-K_{\rM} \in\VV_{2r,(\lambda, \lambda, \cdots, \lambda)}^{\dagger}$ is an ample divisor on $\rM_{\bp}(r, 0, \ba_{c})$. Thus $\rM_{\bp}(r, 0, \ba_{c})$ is a Fano variety. However, in many cases (even for rank two), $\ba_{c}$ lies on a wall, so it is not general and not dominant. To avoid this technical difficulty, we perturb the weight data slightly. Let $\ba_{d}$ be a general small perturbation of $\ba_{c}$ such that the set of walls that we meet while the scaling wall crossing from $\ba_{c}(\epsilon)$ to $\ba_{c}$ is equal to that for the scaling wall crossing from $\ba_{d}(\epsilon)$ to $\ba_{d}$. Since $\rM_{\bp}(r, 0, \ba_{d}(\epsilon)) \cong \Fl(V)^{n}\git_{L_{\ba_{d}(\epsilon)}}\SL_{r}$ and $L_{\ba_{d}(\epsilon)}$ is sufficiently close to the symmetric linearization, by Proposition \ref{prop:maxPicnumGIT}, $\rho(\rM_{\bp}(r, 0, \ba_{d}(\epsilon)) = (r-1)n$ if $n > 2r$. 

We show that if $n$ is sufficiently large, then $\ba_{d}$ is dominant.

\begin{proposition}\label{prop:adisdominant}
Let $\Delta(s, d, \cJ)$ be a wall one meets while the scaling wall-crossing from $\ba_{d}(\epsilon)$ to $\ba_{d}$. Suppose that $n$ is sufficiently large. Then the wall-crossing is not a blow-down. In particular, $\ba_{d}$ is dominant.
\end{proposition}

\begin{proof}
Let $\ba = \ba_{d}(c)$ be the weight on the wall, and $\ba^{\pm}$ are weights near the wall as before. Recall that the wall-crossing center $Y^{-}$ is isomorphic to $\PP\Ext^{1}(\cE^{-}, \cE^{+})$-bundle over $\rM_{\bp}(s, d, \bfb) \times \rM_{\bp}(r-s, -d, \bc)$. 

Suppose that a general point on $Y^{-}$ parametrizes a parabolic bundle with a non-trivial underlying bundle. Then $Y^{-} \subset \Theta$. If the wall-crossing is blow-down, then $Y^{-} = \Theta$ because $\Theta$ is an irreducible divisor. But $\Theta$ is not contracted by scaling-up by Lemma \ref{lem:thetaisnotcontracted}. Thus we may assume that a general point of $Y^{-}$ parametrizes a parabolic bundle with a trivial underlying bundle.

Let $\cE = (E, \{W_{\bullet}^{i}\}, \ba^{-})$ be a general point on $Y^{-}$ and $\cE^{+} = (E^{+}, \{{W|_{E}}_{\bullet}^{i}\}, \bfb^{-})$ (resp. $\cE^{-} = (E^{-}, \{{W/E}_{\bullet}^{i}\}, \bc^{-})$) be the destabilizing subbundle (resp. quotient bundle). Since $E$ is trivial, $E^{+}$ (resp. $E^{-}$) is a direct sum of line bundles with nonpositive (resp. nonnegative) degrees. Thus $\dim \Hom(E^{+}\otimes \cO(-(n-2)), E^{-}) = -dr +(n-1)s(r-s)$. By Proposition \ref{prop:extdim}, with respect to $\ba^{-}$,
\[
	\dim \Ext^{1}(\cE^{-}, \cE^{+})
	= -dr - s(r-s) + \sum_{i=1}^{n}\dim \omega_{J^{i}}.
\]

If the wall-crossing is a blow-down,
\[
	\dim \Ext^{1}(\cE^{-}, \cE^{+}) + \dim \rM_{\bp}(s, d, \bfb)
	+ \dim \rM_{\bp}(r-s, -d, \bc) =
	\dim \rM_{\bp}(r, 0, \ba).
\]
Since $\dim \rM_{\bp}(r, d, \ba) = nr(r-1)/2 - r^{2}+1$, this is equivalent to
\[
	\sum_{i=1}^{n} \dim \omega_{J^{i}} = (n-1)s(r-s) + dr - 1.
\]

Now we show that such a wall does not appear when $n \gg 0$. If $\Delta(s, d, \cJ)$ is a wall that we cross while scaling, then there is a constant $0 < c \le 1$ such that
\[
	\mu(\cE^{+}) = \mu(\cE)
\]
for the weight $\ba = \ba_{c}(c)$ on $\Delta(s, d, \cJ)$.

Note that the weight data $\ba_{c}(c)$ is defined as $a_{c}(c)_{\bullet}^{i} = \frac{c}{r}(r-1, r-2, \cdots, 1)$. Thus
\[
\begin{split}
	\mu(\cE^{+})
	&= \frac{1}{s}\left(d + \sum_{i=1}^{n}\sum_{j \in J^{i}}\frac{c}{r}(r-j)
	\right)
	= \frac{1}{s}\left(d + \sum_{i=1}^{n}\sum_{k = 1}^{s}
	\left(c -\frac{c}{r}k - \frac{c}{r}(J_{k}^{i} - k)\right)\right)\\
	&= \frac{1}{s}\left(d + cn\left(s - \frac{s(s+1)}{2r}\right)
	- \frac{c}{r}\sum_{i=1}^{n}\dim \omega_{J^{i}}\right)\\
	&= \frac{1}{s}\left(d + cn\left(s - \frac{s(s+1)}{2r}\right)
	- \frac{c}{r}\left((n-1)s(r-s) + dr -1)\right)\right)\\
	& = \frac{1}{s}\left(cn\left(s - \frac{s(s+1)}{2r} - \frac{s(r-s)}{r}
	\right) + \frac{c}{r}(1+s(r-s)) + (1-c)d\right).
\end{split}
\]
On the other hand,
\[
	\mu(\cE) = \frac{1}{r}\left(\sum_{i=1}^{n}\sum_{j=1}^{r-1}
	\frac{c}{r}(r-j)\right)
	= \frac{cn(r-1)}{2r}.
\]
From $\mu(\cE^{+}) = \mu(\cE)$, we have
\[
	\frac{csn(r-1)}{2r} = cn\left(s - \frac{s(s+1)}{2r}
	- \frac{s(r-s)}{r}\right)
	+ \frac{c}{r}(1+s(r-s))+(1-c)d,
\]
which is equivalent to
\[
	c\left(ns\frac{s-r}{2r} + \frac{1+s(r-s)}{r}\right) = -(1-c)d.
\]
If $n \gg 0$, then the left hand side is a negative number, but the right hand side is non-negative because $d < 0$. Thus there is no such $0 < c \le 1$.
\end{proof}

We are ready to prove the main theorem.

\begin{proof}[Proof of Theorem \ref{thm:Fanotype}]
First of all, suppose that $n$ is sufficiently large. By Proposition \ref{prop:adisdominant}, $\rM := \rM_{\bp}(r, 0, \ba_{d})$ has Picard number $(r-1)n + 1$. Then, $-K_{\rM}$ is nef because it is a limit of ample divisors. If the anticanonical divisor is not big, then the wall-crossing center is the whole $\rM$, and $\dim \Ext^{1}(\cE^{-}, \cE^{+}) = \dim \rM - \dim \rM_{\bp}(s, d, \bfb) - \dim \rM_{\bp}(r-s, -d, \bc) + 1$, or equivalently, $\sum_{i=1}^{n} \dim \omega_{J^{i}} = (n-1)s(r-s) + dr$. By a similar computation as in the proof of Proposition \ref{prop:adisdominant}, one can check that such a boundary wall-crossing does not occur as long as $n$ is large. Thus the anticanonical divisor is also big and $\rM$ is a smooth (Proposition \ref{prop:smooth}) weak Fano variety. Thus $\rM$ is of Fano type.

For a general non-necessarily dominant weight $\ba$, because $\ba_{d}$ is dominant, $\rM_{\bp}(r, 0,\ba)$ is obtained from $\rM_{\bp}(r, 0, \ba_{d})$ by taking several flips and blow-downs, but no blow-ups. By \cite[Theorem 1.1, Corollary 1.3]{GOST15}, $\rM_{\bp}(r, 0, \ba)$ is also of Fano type. When $n$ is small, by Lemma \ref{lem:extension}, $\rM_{\bp}(r, 0, \ba)$ is an image of $\rM_{\bq}(r, 0, \ba')$ for some large $\bq$. Thus it is of Fano type by \cite[Corollary 1.3]{GOST15}.
\end{proof}


\section{Mori's program of the moduli space}\label{sec:MMP}

We are ready to run Mori's program of $\rM_{\bp}(r, 0, \ba)$. In this section, $n > 2r$ and $\ba$ is a dominant weight.

\subsection{Birational models}\label{ssec:birationalmodels}
Recall that for an integral divisor $D$ on a projective variety $X$,
\[
	X(D) := \proj \bigoplus_{m \ge 0}\rH^{0}(X, \cO(mD))
\]
be the associated projective model. The following observation is an immediate consequence of Pauly's theorem (Theorem \ref{thm:Pauly}).

\begin{proposition}\label{prop:birationalmodels}
Let $D \in \mathrm{int}\Eff(\rM_{\bp}(r, 0, \ba))$. Then $\rM_{\bp}(r, 0, \ba)(D) \cong \rM_{\bp}(r, 0, \bfb)$ for some parabolic weight $\bfb$.
\end{proposition}

In particular, all \emph{birational} models of $\rM_{\bp}(r, 0, \ba)$ obtained from Mori's program are again moduli spaces of parabolic bundles with some weight data.

\begin{proof}
For a notational simplicity, set $\rM = \rM_{\bp}(r, 0, \ba)$. We may assume the $\rM$ is the blow-up of $\Fl(V)^{n}\git_{L}\SL_{r}$ along $\rM_{\bp}(r-1, -1, \bfb)$ or $\rM_{\bp}(r-1, 1, \bc)$. Let $\pi : \rM \to \Fl(V)^{n}\git_{L}\SL_{r}$ be the blow-up morphism, and $Y^{+} = \Theta$ be the exceptional divisor. With respect to such an $L$, by Proposition \ref{prop:maxPicnumGIT}, $\mathrm{Pic}(\Fl(V)^{n}\git_{L}\SL_{r})$ is identified with an index $r$ sublattice of $\mathrm{Pic}(\Fl(V)^{n})$. Thus any line bundle on $\Fl(V)^{n}\git_{L}\SL_{r}$ can be uniquely written as $\otimes_{i=1}^{n}\pi_{i}^{*}\overline{F}_{\lambda^{i}}$ where $\overline{F}_{\lambda^{i}}$ is the descent of $F_{\lambda^{i}}$ on $\Fl(V)$ and $\pi_{i} : \Fl(V)^{n} \to \Fl(V)$ is the $i$-th projection. Similarly, any line bundle $\cO(D)$ on $\rM$ can be uniquely written as $\pi^{*}\otimes_{i=1}^{n}\pi_{i}^{*}\overline{F}_{\lambda^{i}} - k\Theta$ for some $k \in \ZZ$.

When $k = 0$, 
\[
\begin{split}
	\rM(D) &= (\Fl(V)^{n}\git_{L}\SL_{r})
	(\otimes_{i=1}^{n}\pi_{i}^{*}\overline{F}_{\lambda^{i}})
	= \proj \bigoplus_{m \ge 0}\rH^{0}(\Fl(V)^{n}\git_{L}\SL_{r},
	\otimes_{i=1}^{n}\pi_{i}^{*}\overline{F}_{\lambda^{i}}^{m})\\
	&= \proj \bigoplus_{m \ge 0}\rH^{0}(\Fl(V)^{n},
	\otimes_{i=1}^{n}\pi_{i}^{*}F_{\lambda^{i}}^{m})^{\SL_{r}}
	= \Fl(V)^{n}\git_{\otimes_{i=1}^{n}\pi_{i}^{*}F_{\lambda^{i}}}\SL_{r},
\end{split}
\]
which is $\rM_{\bp}(r, 0, \bfb)$ for some $\bfb$ by Theorem \ref{thm:smallweightGIT}.

If $k < 0$, then $\rM(D) = \rM(\pi^{*}\otimes_{i=1}^{n}\pi_{i}^{*}\overline{F}_{\lambda^{i}}-k\Theta) = \rM(\pi^{*}\otimes_{i=1}^{n}\pi_{i}^{*}\overline{F}_{\lambda^{i}})$ because $\Theta$ is an exceptional divisor of the rational contraction $\rM \dashrightarrow \rM(\pi^{*}\otimes_{i=1}^{n}\pi_{i}^{*}\overline{F}_{\lambda^{i}})) = \Fl(V)^{n}\git_{\otimes_{i=1}^{n}\pi_{i}^{*}F_{\lambda^{i}}}\SL_{r}$.

Suppose that $k > 0$.
\[
\begin{split}
	\rH^{0}(\rM, \pi^{*}\otimes_{i=1}^{n}\pi_{i}^{*}
	\overline{F}_{\lambda^{i}}) &=
	\rH^{0}(\Fl(V)^{n}\git_{L}\SL_{r}, \otimes_{i=1}^{n}
	\pi_{i}^{*}\overline{F}_{\lambda^{i}})
	= \rH^{0}(\Fl(V)^{n}, \otimes_{i=1}^{n}
	F_{\lambda^{i}})^{\SL_{r}}\\
	&= \VV_{N,(\lambda^{1}, \lambda^{2}, \cdots, \lambda^{n})}^{\dagger}
\end{split}
\]
for some $N > 0$. Thus $\rH^{0}(\rM, D) = \VV_{N-k,(\lambda^{1}, \lambda^{2}, \cdots, \lambda^{n})}^{\dagger}$. If $N - k > \lambda_{1}^{i}$ for all $i$, then Theorem \ref{thm:Pauly} implies that $\VV_{N-k,(\lambda^{1}, \lambda^{2}, \cdots, \lambda^{n})}^{\dagger}$ is an ample linear system on $\rM_{\bp}(r, 0, \bfb)$ for some $\bfb$.

Suppose that $N - k = \lambda_{1}^{i}$ for some $i$. Then for any $m$,
\[\rH^{0}(mD-\Theta) = \VV_{m(N-k)-1,(m\lambda^{1}, m\lambda^{2}, \cdots, m\lambda^{n})}^{\dagger} = 0\] because $m(N-k)-1 < m\lambda_{1}^{i}$. Thus $D$ is on the boundary of the effective cone.
\end{proof}

We close this section with a lemma which was used in the proof of Proposition \ref{prop:adisdominant}. 

\begin{lemma}\label{lem:thetaisnotcontracted}
During a scaling wall-crossing, $\Theta$ is not contracted. 
\end{lemma}

\begin{proof}
Let $\ba$ be a dominant weight. We may assume that $\rM := \rM_{\bp}(r, 0, \ba)$ is the blow-up of $\Fl(V)^{n}\git_{L}\SL_{r}$. By Proposition \ref{prop:birationalmodels}, the scaling wall-crossing is the computation of $\rM(D - c\Theta)$ where $D = \pi^{*}\otimes_{i=1}^{n}\pi_{i}^{*}\overline{F}_{\lambda^{i}}$, from $c = 0$ to $c \gg 0$. 

Suppose that for some $c > 0$, $\rM(D - c\Theta)$ is a contraction of $\Theta$. Then $\rM(D - c\Theta + d\Theta) = \rM(D - c\Theta)$ for any $d > 0$. In particular, $\rM(D - \epsilon \Theta)$ is a contraction of $\Theta$ for $0 < \epsilon \ll 1$. But $D - \epsilon \Theta$ is an ample divisor on $\rM$ and we have a contradiction.
\end{proof}

\subsection{Effective cone}\label{ssec:effectivecone}

The first step of Mori's program is the computation of the effective cone.

For some weight data $\bfb$,  $\rM_{\bp}(r, 0, \bfb)$ may be empty. By combining this observation with Proposition \ref{prop:birationalmodels}, we can compute an H-representation of $\Eff(\rM_{\bp}(r, 0, \ba))$. This result was obtained by Belkale in \cite{Bel08b} in a greater generality and with a different idea.

Set $\rM := \rM_{\bp}(r, 0, \ba)$. Since $\rM$ is a Mori dream space, $\Eff(\rM)$ is a closed polyhedral cone. For each $D \in \Eff(\rM)$, $\rH^{0}(D)$ is identified with $\VV_{\ell,(\lambda^{1}, \lambda^{2}, \cdots, \lambda^{n})}^{\dagger}$. There are two classes of linear inequalities for the non-vanishing of conformal blocks:
\begin{enumerate}
\item $\lambda_{j}^{i} \ge \lambda_{j+1}^{i}$ (it includes $\lambda_{r-1}^{i} \ge 0$ by our normalization assumption);
\item $\lambda_{1}^{i} \le \ell$.
\end{enumerate}
The first class of inequalities comes from the effectiveness of $\otimes_{i=1}^{n}F_{\lambda^{i}}$ on $\Fl(V)^{n}$. For the second class, see \cite[Section 4]{BGM15} for an explanation. Here we construct extra linear inequalities.

Recall that the (genus zero) \emph{Gromov-Witten invariant} counts the number of rational curves intersecting several subvarieties. Here we employ the definition in \cite{Ber97}, which is slightly different from the standard definition using moduli spaces of stable maps (\cite{FP97}). For a partition $\lambda = (r \ge \lambda_{1} \ge \lambda_{2} \ge \cdots \ge \lambda_{s}\ge 0)$ and a complete flag $W_{\bullet}$ of an $r$-dimension vector space $V$, we obtain a Schubert subvariety $\Omega_{\lambda}(W_{\bullet}) \subset \Gr(s, V) = \Gr(s, r)$. Its numerical class is independent of the choice of $W_{\bullet}$, and is denoted by $\omega_{\lambda} \in \rH^{*}(\Gr(s, r))$. For a collection of general complete flags $W_{\bullet}^{i}$ of $V$ and a nonnegative integer $d$, the Gromov-Witten invariant
\[
	\langle \omega_{\lambda^{1}}, \omega_{\lambda^{2}}, \cdots,
	\omega_{\lambda^{n}}\rangle_{d}
\]
is the number of maps $f : (\PP^{1}, \bp = (p^{i})) \to \Gr(s, r)$ of degree $d$ such that $f(p^{i}) \in \Omega_{\lambda^{i}}(W_{\bullet}^{i})$ if the number is finite, and otherwise it is zero. Since the moduli space of maps from $\PP^{1}$ to $\Gr(s, r)$ is not proper, a rigorous definition requires a compactified space of maps, for instance the quot scheme over $\PP^{1}$, but by Moving Lemma (\cite[Lemma 2.2A]{Ber97}), the number is equal to the number of genuine maps from $\PP^{1}$ to $\Gr(s, r)$.

\begin{proposition}\label{prop:facetGWinv}
For each collection of partitions $\lambda^{1}, \lambda^{2}, \cdots, \lambda^{n}$ of length $s$ and a nonpositive integer $d$ such that the Gromov-Witten invariant $\langle \omega_{\lambda^{1}}, \omega_{\lambda^{2}}, \cdots, \omega_{\lambda^{n}}\rangle_{-d}$ on $\Gr(s, r)$ is one, there is a linear inequality 
\begin{equation}\label{eqn:supportingplane}
	\frac{1}{s}
	\left(d\ell + \sum_{i=1}^{n}\sum_{j \in J^{i}}\lambda_{j}^{i}\right)
	\le
	\frac{1}{r}\left(\sum_{i=1}^{n}\sum_{j = 1}^{r-1}\lambda_{j}^{i}\right).
\end{equation}
which defines $\Eff(\rM)$. Moreover, these inequalities and two classes (1) and (2) of inequalities provide the H-represenation of $\Eff(\rM)$.
\end{proposition}

\begin{proof}
Let $D$ be a general point on a facet of $\Eff(\rM)$, which is not one of facets described above. Take an embedding of a small line segment $\gamma : [-\epsilon, \epsilon] \to \rN^{1}(\rM)$ such that $\gamma(0) = D$ and $\gamma(x) \in \mathrm{int}\; \Eff(\rM)$ when $x < 0$. Let $D^{\pm} := \gamma(\pm\epsilon)$.

Note that each $x \in [-\epsilon, \epsilon]$ defines an $\RR$-divisor $D_{x} = \cL^{\ell} \otimes \otimes_{i=1}^{n}F_{\lambda^{i}}$, and hence a parabolic weight $\ba_{x}$ by setting $(a_{x})_{i} = \frac{1}{\ell}(\lambda_{j}^{i})$. We may assume that all $\ba_{x}$ are general except $\ba_{0}$. Because the moduli space becomes empty after changing the weight from $\ba_{-\epsilon}$ to $\ba_{\epsilon}$, there is a boundary wall $\Delta(s, d, \cJ)$ at $\ba_{0}$. A wall-crossing is a boundary one if and only if a general point $\cE = (\cO^{r}, \{W_{\bullet}^{i}\}, \ba_{-\epsilon})$ of $\rM_{\bp}(r, 0, \ba_{-\epsilon})$ has the unique detabilizing bundle $\cE^{+} = (E^{+}, \{{W|_{E^{+}}}_{\bullet}^{i}\}, \bfb_{-\epsilon})$ of rank $s$ such that $\mu(\cE^{+}) = \mu(\cE)$ with respect to $\ba_{0}$. This implies that there is a short exact sequence
\[
	0 \to E^{+} \to E \to E^{-} \to 0
\]
of bundles such that $E^{+}|_{p^{i}} \in \Omega_{\lambda^{i}}(W_{\bullet}^{i})$. Therefore there is a map $f : (\PP^{1}, \bp) \to \Gr(s, r)$ of degree $-d$ such that $f(p^{i}) \in \Omega_{\lambda^{i}}(W_{\bullet}^{i})$. Therefore
\[
	\langle \omega_{\lambda^{1}}, \omega_{\lambda^{2}}, \cdots,
	\omega_{\lambda^{n}}\rangle_{-d} = 1.
\]
In particular, to have a nonempty moduli space, $\mu(\cE^{+}) \le \mu(\cE)$, which is \eqref{eqn:supportingplane}. 

Now suppose that $D$ is a divisor satisfies all of the given strict linear inequalities of the form \eqref{eqn:supportingplane} for every collection of partitions $\lambda^{1}, \lambda^{2}, \cdots, \lambda^{n}$ with $\langle \omega_{\lambda^{1}}, \omega_{\lambda^{2}}, \cdots, \omega_{\lambda^{n}}\rangle_{-d} = 1$. Let $\ba$ be the associated parabolic weight data. Then for a general parabolic bundle $\cE = (\cO^{r}, \{W_{\bullet}^{i}\}, \ba)$, there is no possible destabilizing bundle. Therefore $\cE \in \rM_{\bp}(r, 0, \ba)$ and the moduli space is nonempty. Because $D$ is an ample divisor on $\rM_{\bp}(r, 0, \ba)$, $|mD| \ne \emptyset$ for some $m > 0$. Therefore $D \in \mathrm{int} \Eff(\rM)$. By taking the closure, we can obtain the effective cone. 
\end{proof}

\begin{remark}
The computation of the V-representation from the H-representation is highly nontrivial. In \cite{Bel17}, Belkale explains how to compute the extremal rays of the effective cone for the quotient stack $[\Fl(V)^{3}/\SL_{r}]$. He informed to the authors that this computation can be generalized to the case of arbitrary $n$  and for $\rM_{\bp}(r, 0, \ba)$, too.
\end{remark}

\subsection{Projective models and wall-crossing}\label{ssec:projmodel}

The remaining steps of Mori's program are the computation of projective models $\rM(D)$ for $\rM := \rM_{\bp}(r, 0, \ba)$ and the study of the rational contraction $\rM \dashrightarrow \rM(D)$. For $D \in \mathrm{int}\Eff(\rM)$, Proposition \ref{prop:birationalmodels} already provides the answer. It remains to find projective models associated to $D \in \partial \Eff(\rM)$. We content ourselves with a description for facets of $\partial \Eff(\rM)$.

The first type of facets are that associated to Gromov-Witten invariants, as described in Section \ref{ssec:effectivecone}. We call this type of facets as \emph{GW facets}. In this case, the boundary wall-crossing in Proposition \ref{prop:boundarywall} gives the contraction.

\begin{proposition}
Suppose that $D$ is a general point on a GW facet of $\Eff(\rM)$. Then $\rM(D) = \rM_{\bp}(s, -d, \bfb) \times \rM_{\bp}(r-s, d, \bc)$ for some $0 < s < r$, $d \ge 0$, and $\bfb$ and $\bc$.
\end{proposition}

The second type of facets are of the form $\lambda_{j}^{k} = \lambda_{j+1}^{k}$. This case is related to moduli spaces of parabolic bundles with degenerated flags, which forgets $j$-th flag on $p^{k}$. In \cite{Pau96}, Pauly proved Theorem \ref{thm:Pauly} for such degenerated flags, too. The proof of the next proposition is essentially same to that of Proposition \ref{prop:birationalmodels}.

\begin{proposition}
Suppose that $D$ is a general point of the facet of $\Eff(\rM)$ which is given by $\lambda_{j}^{k} = \lambda_{j+1}^{k}$. Then $\rM(D) = \rM_{\bp}(r, 0, \bfb)$, which is the moduli space of parabolic bundles where its $k$-th flag is a partial flag of type $(1, 2, \cdots, \hat{j}, \cdots, r-1)$.
\end{proposition}

The last type of facets are of the type $\lambda_{1}^{k} = \ell$.

\begin{proposition}\label{prop:projmodel3}
Suppose that $D$ is a general point on the facet $\lambda_{1}^{k} = \ell$. Then $\rM(D) = \rM_{\bp}(r, -1, \bfb)$ where $\bfb$ is a parabolic weight such that $b^{i} = \frac{1}{\ell}(\lambda^{i}_{1}, \lambda^{i}_{2}, \cdots, \lambda^{i}_{r-1})$ for $i \ne k$ and $b^{k} = \frac{1}{\ell}(\lambda^{k}_{1} - \lambda^{k}_{r-1}, \lambda^{k}_{2} - \lambda^{k}_{r-1}, \cdots, \lambda^{k}_{r-2} - \lambda^{k}_{r-1})$ (the last flag is of type $(2, 3, \cdots, r-1)$).
\end{proposition}

\begin{proof}
By symmetry, we may assume that $k = n$. Let $D'$ be a big divisor which is sufficiently close to $D$. Then $\rM(D) = \rM(D')(D)$. Thus we may replace $\rM$ by $\rM(D')$. Equivalently, after Theorem \ref{thm:Pauly}, we may assume that $\ba$ is sufficiently close to $(\frac{1}{\ell}\lambda^{i})$. 

Let $\cE = (E, \{W_{\bullet}^{i}\}, \ba) \in \rM_{\bp}(r, 0, \ba)$. Consider the quotient map $E \to E|_{p^{n}}/W_{r-1}^{n} \to 0$ and let $E'$ be the kernel. Then $E'$ is a vector bundle of degree $-1$. For $p^{i}$ with $i < n$, let ${W'}_{j}^{i} = W_{j}^{i}$. Over $p^{n}$, let ${W'}_{j}^{n} = r^{-1}(W_{j}^{n})$ where $r : E'|_{p^{n}} \to E|_{p^{n}}$ is the restriction of $E' \hookrightarrow E$. Note that $\dim {W'}_{j}^{n} = j+1$. Thus we have a quasi parabolic bundle $\cE' := (E', \{{W'}_{\bullet}^{i}\})$ whose last flag over $p^{n}$ is of type $(2, 3, \cdots, r-1)$.

We claim that $\cE'$ is semistable with respect to $\bfb$. Let $\cF' = (F', \{{V'}_{\bullet}^{i}\}, \bc)$ be a parabolic subbundle of $\cE'$. To avoid a confusion, the slope with respect to $\bfb$ is denoted by $\mu_{\bfb}$. Because $b_{j}^{n} = a_{j}^{n} - a_{r-1}^{n}$,
\[
\begin{split}
	\mu_{\bfb}(\cE') &= \frac{1}{r}\left(-1
	+ \sum_{i < n}\sum_{j=1}^{r-1}
	a_{j}^{i} + b_{1}^{n} + \sum_{j=1}^{r-2}b_{j}^{n}\right)
	= \mu(\cE) - \frac{1}{r} - \frac{1}{r}\sum_{j=1}^{r-1}a_{j}^{n}
	+ \frac{1}{r}\left(b_{1}^{n}+\sum_{j=1}^{r-2}b_{j}^{n}\right)\\
	&= \mu(\cE) - \frac{1}{r} + \frac{1}{r}a_{1}^{n} - a_{r-1}^{n}.
\end{split}
\]

Suppose that $F'$ is a rank $s$ subbundle of $E$. Then $\ker r \cap F'|_{p^{n}} = 0$. If $J^{i}$ denotes the subset of indices $j\in[r]$ such that ${W'}_{j}^{i} \cap F'|_{p^{i}} \neq {W'}_{j-1}^{i} \cap F'|_{p^{i}}$,
\begin{equation}\label{eqn:slope}
\begin{split}
	\mu_{\bfb}(\cF') &= \frac{1}{s}\left(\deg F'
	+ \sum_{i < n}\sum_{j \in J^{i}}a_{j}^{i}
	+ \sum_{j \in J^{n}}b_{j}^{n}\right)\\
	&= \frac{1}{s}\left(\deg F' + \sum_{i}\sum_{j \in J^{i}}a_{j}^{i}\right)
	+ \frac{1}{s}\left(\sum_{j \in J^{n}}(b_{j}^{n} - a_{j}^{n})\right)
	= \mu(\cF') - a_{r-1}^{n}.
\end{split}
\end{equation}
Therefore $\mu_{\bfb}(\cE') - \mu_{\bfb}(\cF') = \mu(\cE) - \mu(\cF') - \frac{1}{r} + \frac{1}{r}a_{1}^{n}$. Since $a_{1}^{n}$ is sufficiently close to one, $\mu_{\bfb}(\cE') - \mu_{\bfb}(\cF') \ge 0$.

If $F'$ is not a subbundle of $E$, then there is a subbundle $F$ of $E$ which contains $F'$ and $F/F'$ is a torsion sheaf. Since $E/E'$ is of length one, $F/F'$ is also of length one. Therefore $\deg F' = \deg F - 1$. Let $\cF$ be the parabolic subbundle of $\cE$ whose underlying bundle is $F$. By a similar computation with \eqref{eqn:slope}, we have
\[
	\mu_{\bfb}(\cF') = \mu(\cF) - \frac{1}{s} + \frac{1}{s}(a_{1}^{n}
	- a_{r-1}^{n}) - a_{r-1}^{n}
\]
and it is straightforward to see that $\mu_{\bfb}(\cE') - \mu_{\bfb}(\cF') > 0$.

Therefore the map $\rM \to \rM_{\bp}(r, -1, \bfb)$ which sends $\cE$ to $\cE'$ is a well-defined morphism. This is a $\PP^{1}$-fibration and it is of relative Picard number one. Thus it is associated to a facet of $\Eff(\rM)$ which intersects the closure of the chamber for $\rM$. If $\rM(D) \ne \rM_{\bp}(r, -1, \bfb)$, then $\rM_{\bp}(r, -1, \bfb)$ is associated to another facet. Because we already know the projective models for the other facets, the only remaining possibility is $\lambda_{1}^{i} = \ell$ for some $i \ne n$. But because of the existence of a wall of the type $\Delta(s, 0, \cJ)$ where $1 \in J^{n}$ but $1 \notin J^{i}$, only one of these facets is a facet of the closure of the chamber of $\rM$. Thus $\rM(D) = \rM_{\bp}(r, -1, \bfb)$.
\end{proof}

Because all of the birational models and projective models can be described in terms of moduli spaces of parabolic bundles, the wall-crossings in Section \ref{ssec:wallchamber} are building blocks of the rational contraction $\rM \dashrightarrow \rM(D)$. This is in some sense very satisfactory, because all of them are smooth blow-ups/downs and projective bundle morphisms.

\subsection{Rationality}\label{ssec:rationality}

It is an old open problem determining whether the moduli space of (parabolic) bundles with a fixed determinant over a Riemann surface is rational or not (\cite{KS99, Hof07}). The wall-crossing toward a boundary wall was applied to show the fact that $\rM_{\bp}(r, 0, \ba)$ is rational in \cite{BH95}. Here we leave a sketch, for a reader's convenience.

It is sufficient to prove for the case that $\ba$ is sufficiently small, so $\rM_{\bp}(r, 0, \ba) = \Fl(V)^{n}\git_{L}\SL_{r}$ for some $L$. Cross several walls of type $\Delta(s, 0, \cJ)$, which are indeed GIT walls. If $\ba$ is sufficiently close to the boundary, then by Proposition \ref{prop:boundarywall}, $\rM_{\bp}(r, 0, \ba)$ is a projective bundle over $\rM_{\bp}(s, 0, \bfb) \times \rM_{\bp}(r-s, 0, \bc)$. Thus the problem is reduced to a lower rank case. If $r = 1$, the moduli space is a point, so it is trivial.

\begin{proposition}[\protect{\cite[Proposition 5.1]{BH95}}]
The moduli space $\rM_{\bp}(r, 0, \ba)$ is rational.
\end{proposition}


\bibliographystyle{alpha}

\end{document}